\documentclass[twoside,11pt]{article}

 \usepackage[abbrvbib, preprint]{jmlr2e}

\usepackage{jmlr2e}

\usepackage{amssymb,latexsym,amsmath}
\usepackage{epsfig}
\usepackage{caption}
 \usepackage{float}

 \usepackage{amsmath}
\usepackage{amsfonts}
\usepackage{array}
\usepackage{dsfont}
\usepackage{hyperref}
\usepackage{amssymb}

 \usepackage{algorithm}
\usepackage{algorithmic}
\usepackage{pifont}

\usepackage{makeidx}
\usepackage{graphicx}
\graphicspath{ {./images/} }

\allowdisplaybreaks

 \usepackage{graphicx,fancybox,latexsym,epsfig}
\usepackage{fancyhdr,amsmath,times,amsxtra,amssymb}
\usepackage{color}

\newenvironment{mainresult}{
  \par\medskip
  \noindent\textbf{Main Result.}\ \itshape
}{
  \par\medskip\normalfont
}

\def\supp{\mathop{\rm supp}}

\def\argmin{\mathop{\rm arg\, min}}

\def\B{{\mathcal B}}

\def\E{{\mathcal E}}

\def\P{{\mathcal P}}

\newcommand{\R}{\mathds{R}}

\newcommand{\N}{\mathbb{N}}

\newcommand{\dd}{\mathrm{d}}

\newcommand{\sy}[1]{{\color{black} #1}}

\newcommand{\adk}[1]{{\color{black} #1}}

\allowdisplaybreaks

\ShortHeadings{Q-Learning under Average Cost Criteria }{Kara and Y\"uksel}
\firstpageno{1}

\begin{document}
\sloppy
\title{Approximations and Learning for Continuous State and Action MDPs under Average Cost Criteria \thanks{
This research was supported in part by
the Natural Sciences and Engineering Research Council (NSERC) of Canada.}
}

\author{Ali Devran Kara and Serdar Y\"uksel
\thanks{Ali D. Kara is with the Department of Mathematics, Florida State University, Tallahassee, FL, USA, Email: akara@fsu.edu. S. Y\"uksel is with the Department of Mathematics and Statistics,
     Queen's University, Kingston, ON, Canada,
     Email: yuksel@queensu.ca}
     }

\editor{}

\maketitle

\begin{abstract}
 In this paper, for Markov Decision Processes (MDPs) with standard Borel spaces, (i) we first provide a discretization based approximation method for MDPs with continuous spaces under average cost criteria, and provide error bounds for approximations when the dynamics are only weakly continuous (for asymptotic convergence of errors as the grid sizes vanish) or Wasserstein continuous (with a rate in approximation as the grid sizes vanish) under certain ergodicity assumptions. In particular, we relax the total variation condition given in prior work to weak continuity or  Wasserstein continuity. (ii) We provide synchronous and asynchronous (quantized) Q-learning algorithms for continuous spaces via quantization (where the quantized state is taken to be the actual state in corresponding Q-learning algorithms presented in the paper), and establish their convergence.  (iii) We finally show that the convergence is to the optimal Q values of a finite approximate model constructed via quantization, which implies near optimality of the arrived solution. 
 \end{abstract}


\section{Introduction}\label{section:intro}

In this paper, we study approximate solutions for Markov decision processes (MDPs) under  average cost criterion. We consider problems with continuous state and action spaces and provide approximate planning and reinforcement learning results. In particular, we show near-optimality of solutions to approximate models and those obtained via reinforcement learning. 

Before we present the related research in these problems and discuss our contributions in more detail, we introduce the problem formulation:
A fully observed Markov control model is a tuple \[(\mathbb{X},\mathbb{U},  {\cal T}, c),\]
where $\mathbb{X}$ is the (standard Borel) state space that is a metric space with the associated metric $d_\mathds{X}$. Similarly, $\mathbb{U}$ is the action space, assumed to be a metric space with the metric $d_\mathds{U}$. Under $d_\mathds{X}$ and $d_{\mathds{U}}$, $\mathds{X}$ and $\mathds{U}$ are complete and separable. The transition kernel of the model is denoted by ${\cal T}$ on $\mathbb{X}$ given $\mathbb{X}\times\mathds{U}$. Finally, $c: \mathbb{X}\times\mathds{U} \to \mathbb{R}$ is the cost  function. 

Let $\mathbb{H}_0:= \mathbb{X}, \quad \mathbb{H}_t = \mathbb{H}_{t-1} \times \mathbb{X\times U}$ for $t=1, 2, \ldots$. We let, for $t \in \mathbb{Z}_+$, $h_t$ denote an element of $\mathbb{H}_t$, with $h_t=\{x_{[0,t]},u_{[0,t-1]}\}$.  We use the notation $x_{[0,t]}:=(x_0,x_1,\cdots,x_t)$. 

An admissible control policy $\pi$ is a sequence of measurable functions $(\pi_t: t = 0,1,2\dots)$ such that $\pi_t: \mathbb{H}_t \to \P(\mathbb{U})$ with $u_t=\pi_t(h_t)$ where $\P(\mathds{U})$ denotes the set of all probability measures on $\mathds{U}$. We denote the set of all admissible policies by $\Pi_A$. If an admissible policy $\pi$ is such that $u_t\sim f(x_t)$ for some Borel measurable $f: \mathbb{X} \to \P(\mathbb{U})$, then the policy is said to be stationary; we denote the set of stationary policies by $\Pi_S$. 

We assume that such an MDP model is given on the $(\mathbb{X}\times\mathds{U})^{\mathbb{Z}_+}$-valued state-action process $(X_t, U_t)_{t \in \mathbb{Z}_+}$. For each $x \in \mathbb{X}$ and $\pi \in \Pi_A$, the kernel $\mathcal{T}$ and the policy $\pi$ induce a probability measure (called a strategic measure) $P_x^\pi$ on $((\mathbb{X}\times\mathds{U})^{\mathbb{Z}_+},{\cal B}((\mathbb{X}\times\mathds{U})^{\mathbb{Z}_+}))$ (where ${\cal B}(\cdot)$ denotes the Borel $\sigma$-field on the space $\cdot$) with $X_0 \sim \delta_x$. We let $E_x^\pi$ denote the corresponding expectation.

We consider the following average cost problem:
\begin{eqnarray}\label{AverageCostProblemDef}
J^*(x):= \inf_{\pi} J(x,\pi) = \inf_{\pi \in \Pi_A} \limsup_{T \to \infty} {1 \over T} E^{\pi}_x [\sum_{t=0}^{T-1} c(x_t,u_t)].
\end{eqnarray}

We also define the discounted cost for some discount factor $0<\beta<1$:
\begin{eqnarray}\label{DiscCostProblemDef}
J_\beta^*(x):= \inf_{\pi} J_\beta(x,\pi) = \inf_{\pi \in \Pi_A} \sum_{t=0}^{\infty}\beta^tE^{\pi}_x [ c(x_t,u_t)].
\end{eqnarray}


\subsection{Literature Review}  Finding efficient solutions to MDPs with continuous spaces is a challenging and important problem. In this context, various approximation techniques have been presented \citep[see e.g.][]{DuPr12,Ber75,chow1991optimal,BertsekasTsitsiklisNeuro,CsabaAlgorithms,tsitsiklis1997analysis,singh1995reinforcement,melo2008analysis,gaskett1999q,CsabaSmart,meyn2022control,SaYuLi15c,SaLiYuSpringer}. Many studies in the literature have focused on either finite-horizon problems or infinite-horizon discounted cost problems, 
using the dynamic programming principle and the contraction properties of Bellman operators induced by the discount factor. For average cost problems, however, the same techniques are not immediately applicable. 

For MDPs with continuous state spaces, existence of optimal solutions has been well studied under the infinite horizon average cost criterion, under either weak continuity of the kernel (in both state and action) or strong continuity (of the kernel in actions for every state), together with various stability/ergodicity assumptions. We refer the reader to the works by \cite{survey,hernandez2012adaptive,HernandezLermaMCP,hernandezlasserre1999further,costa2012average,GoHe95,Veg03,vega2003average,HernandezLermaMCP,Borkar2,FeKaZa12} for comprehensive surveys on optimality results and general solution techniques for optimal control under the average cost criteria. However, research on computational and reinforcement learning methods for average cost optimality still entails open problems, especially for MDPs with continuous spaces.


The primary contributions of the paper concern finite approximations for MDPs with continuous spaces, their near optimality, and reinforcement learning for average cost problems. We will summarize the related research in these areas separately:

\noindent{\bf Approximations for MDPs with continuous spaces.}
For the study of continuous space MDPs, establishing regularity and continuity properties of the value functions is crucial. To this end, one usually needs continuity assumptions on the stage-wise cost functions, and continuity assumptions on the transition models under a suitable metric. Weak convergence metrics, and Wasserstein distances are used frequently for the regularity of the transition models since they are, in general, weaker and less demanding than other metric choices, e.g., total variation and relative entropy (Kullback–Leibler divergence) type distance notions. These are used to establish Lipschitz continuity of the value functions, as well as to establish the consistency of model approximations. The papers by \cite{pirotta2015policy,asadi2018lipschitz,Rachelson2010OnTL,maran2023tight}  study MDPs with Wasserstein continuous transition models under the discounted cost setting. However, one drawback of these papers is that they work on a class of Lipschitz continuous policies which is a suboptimal class in general. For the discounted cost criterion, \cite{saldi2014near,SaldiLinderYukselTAC14,SaYuLi15e,SaYuLi15c,KSYContQLearning} have shown that under only weak continuity conditions for an MDP with standard Borel state and action spaces, finite models obtained by the quantization of the state and action spaces lead to control policies that are asymptotically optimal as the quantization rate increases. They also established convergence rates under Lipschitz continuity assumptions on the model, without requiring the class of policies to be a priori continuous. The authors of the current paper generalized these results to general approximation schemes beyond discretization based approaches under several different continuity assumptions on the transition models such as {\it continuous} weak continuity, set-wise continuity as well as total variation continuity \citep{kara2020robustness,KaraYuksel2021Chapter}.

\noindent{\bf On the weak metrics for approximations and robustness to model mismatch.} 
Related to the discussion above, a salient consideration involving approximations for MDPs with general spaces is that unless convergence in models occurs in a strong sense, such as in total variation uniform over sets and actions, as models approach one another, the induced expected costs may not asymptotically agree under every admissible policy. If convergence is to hold under the Wasserstein metric, such a uniform convergence only would hold if policies were restricted to be Lipschitz a priori  \citep[see e.g.][]{pirotta2015policy,asadi2018lipschitz,Rachelson2010OnTL,maran2023tight}, which is often too restrictive for the approximate models. A detailed study on such model convergence and robustness has been carried out by \cite{kara2020robustness,KaraYuksel2021Chapter}. Under {\it continuous convergence of kernels}, \cite{kara2020robustness} in discrete-time and \cite{pradhan2022near} in continuous-time, established robustness when a control policy designed for an approximate/incorrect model is applied to a true model. \cite{SaYuLi15c,KaraYuksel2021Chapter,KSYContQLearning} present a construction for the approximate models through quantizing the actual model with continuous spaces, which allows for continuity and robustness results with only a weak continuity assumption on the true transition kernel which, in turn, leads to {\it continuous weak convergence} of approximate models as discussed by \cite{KaraYuksel2021Chapter} and thus becomes an important special case of robustness under model convergence. Closely related to such results and continuous convergence, a topology on spaces of probability measures corresponding to laws of stochastic processes, which has been used in a wide variety of contexts in stochastic analysis, is defined as follows: a sequence of stochastic processes is said to converge to another process if their finite-dimensional marginals converge weakly, and their conditional distributions of future variables given the past (viewed as measure-valued stochastic processes) also converge weakly. \cite{aldous1981weak}  has termed this {\it extended weak convergence} and  \cite{hellwig1996sequential} has named it {\it the information topology}; these have recently been shown to be equivalent in discrete-time by \cite{backhoff2019all,pammer2024note}. The {\it adapted Wasserstein metric} \citep[see ][]{bartl2024wasserstein,backhoff2019all,beiglbock2022approximation} has been shown to possess similar robustness properties in a variety of applications \citep{bayraktar2020continuity,julio2020adapted,bartl2023sensitivity}, not unlike the continuous weak convergence notion noted above (see \cite{saldi2025kernel} for a comprehensive review). Please see Section \ref{FinMDPWassM} for further discussion.

\noindent{\bf Approximations for the average cost criterion.}
We note that the analysis of average cost problems is typically more challenging, especially for problems with continuous spaces, as the stability (or the ergodicity) of the problem plays a crucial role. For the average cost criterion, \citet{saldiAverage}, \citet[Theorem 4.14]{SaLiYuSpringer} provide error bounds for finite approximations; however, these results require total variation continuity of the transition models as well as certain mixing conditions. 

In this paper, we relax the total variation continuity condition to weak or Wasserstein continuity. 

\noindent{\bf Reinforcement learning for the average cost criterion.}
There are a number of publications that study reinforcement learning methods for MDPs under average cost criterion. To our knowledge, the majority of these studies focus on finite spaces. 

The papers by \citet{abounadi2001learning,gosavi2004reinforcement} are among the earliest studies that provide convergent learning algorithms based on relative value iteration (as well as stochastic shortest path under a recurrence condition for a given state), and the convergence of these algorithms has been established via the ODE method  by \cite{borkar2000ode} for finite models.

\cite{kondaActorCritic} studied policy improvement and actor-critic methods for continuous space (in particular, Polish space) MDPs under average cost criteria. The approach relies on the linear approximation of the value functions and the parametrization of the policies. The convergence of this method is shown under a uniform minorization assumption over the parametrized policy space \citep[Assumption 4.2]{kondaActorCritic} which is similar to the assumption used in our paper (see Assumption~\ref{mix_kernel}). However, the learned solution is only locally optimal due to the nature of policy parametrization methods.

\cite{ormoneit2002kernel} also focus on reinforcement learning methods for MDPs with continuous state spaces under the average cost criteria. Their method is based on a kernel-based approximate dynamic programming, and convergence of the algorithms is shown under a minorization condition (similar to Assumption \ref{mix_kernel}) for several different averaging kernel functions. However, they impose strong regularity conditions on the transition model for the kernel based methods to work. In particular, it is assumed that the transition kernel admits  a density function with respect to the Lebesgue measure such that $\mathcal{T}(dx_1|x,u)=f_u(x_1,x)\lambda(dx_1)$ where $f_u(x_1,x)$ is strictly positive and uniformly continuous in both variables. This assumption is, in particular, even stronger than total variation continuity of the transition kernel.

Among the relatively more recent studies, the comprehensive paper \cite{wan2021learning} provides convergent off-policy learning algorithms to stabilize the value function estimation for finite models; the convergence proof by \cite{wan2021learning} builds on the ODE method \citep{abounadi2001learning,borkar2000ode} but relaxes some of the conditions in \cite{abounadi2001learning} with regard to the reference term subtracted in each iterate. We should note that the proof method of convergence by \cite{abounadi2001learning} (and thus \cite{wan2021learning}) for both the synchronous and asynchronous Q-learning build on the synchronous update dynamics analysis as these are equivalent under the ODE method.

\citet[Section 5]{abounadi2001learning} note the need for generalizing the analysis to continuous spaces for relative Q learning methods. We also note that the appendix of \cite{wan2021learning} notes the continuous state/action setup as an open problem, which our current paper addresses.

\cite{zhang21q} study a policy improvement method for average-cost (reward) MDPs for finite state-action spaces. \cite{wang23am} focus on robust model-free methods for finite systems where the transition model belongs to an uncertainty set which is constructed under various metrics. \cite{suttle23a} focus on relaxing the exponential mixing assumption for average cost criteria for finite models and provides an actor-critic method. \citet{zhang2021finite} study finite sample guarantees for a synchronous Q-learning algorithm under the average cost criterion. We also note that \cite{yang2019provably} study a convergent actor-critic method under average-cost criteria for continuous models with linear systems and additive Gaussian noise. We  refer the reader to the paper by \cite{CsabaAlgorithms} for a general review on the subject. 

We also note more recent work on average cost MDPs that focuses on sample complexity \citep[see][]{chen2025nonasymptotic,bravo2024stochastic,jin2024feasible,zhang2021finite2}.

A complementary  line of work studies infinite-horizon average cost problems for continuous spaces under  structural assumptions. One such direction involves linear MDPs where the transition models and the cost function are assumed to be within the linear span of known basis (feature) functions. \cite{weiAvgLinear,wuMinimaxAvg} and
more recently  \cite{hongAvgLinear,chaeAvgLinearMixture,hongAvgEfficient} study  linear and linear mixture MDPs via the vanishing discount approach.  For kernel-based function approximation, \cite{vakiliKernelAvg} propose an  algorithm under reproducing kernel Hilbert space assumptions on the value function. In this paper, we do not make  parametric assumptions on the model, and only assume   weak or Wasserstein continuity of the transition kernel on a general standard Borel space, and establish almost sure convergence of (quantized) Q-learning to the optimal value of an explicit finite approximate model with near optimality guarantees.

 Almost all of these papers focus on MDP problems under infinite horizon average cost criteria for finite models, i.e. where the state and action spaces are finite, or impose  parametric (linear/kernel) structure on the model and thus are able to use the Markovian nature for learning.

To this end, in our paper we consider general continuous spaces for which we first construct a discretized model where we present weaker continuity conditions than currently available in the literature. After discretization, the convergence analysis for learning methods requires further adaptations as the discretized states are no longer Markovian. In addition, ergodicity conditions are required to ensure the stability of the associated stochastic iteration algorithms. For continuous models, the ergodicity conditions and the stability analysis differ significantly from those involving finite models. In particular, in addition to obtaining average cost counterparts of the analysis of \cite{KSYContQLearning} \citep[which in turn builds on][]{kara2021convergence}, the paper introduces additional technical methods for the convergence analysis that are new even for finite MDPs, whose treatment in the literature has been restricted to the ODE method as noted earlier.

{\bf Contributions.} In view of the above, we address several open questions in the literature along the following directions:

\begin{itemize}
\item[(i)] [Approximation Results for Average Cost Infinite Horizon Control] In Section \ref{finite_app_sec}, we provide a discretization-based approximation method for fully observed MDPs with continuous spaces under the average cost criteria, and we provide error bounds for the approximations when the dynamics are only weakly continuous under certain ergodicity assumptions. Theorem \ref{weak:mainthm2} provides error bounds for action space discretization.  Theorem  \ref{pol_err} presents near optimality of control policies obtained for  models with discretized state spaces. Notably, we relax the total variation condition given by  \cite{saldiAverage,SaLiYuSpringer}  to weak (Feller) continuity (in Theorem \ref{weakContApprAsy}) or Wasserstein continuity conditions (in Theorem \ref{pol_err}); the former leads only to asymptotic convergence, whereas the latter provides a rate of convergence.

\item[(ii)] [Reinforcement Learning Analysis for Continuous State/Action Models] \adk{ In  Section \ref{q_cont_sec}, we consider what happens if one uses Q-learning with a quantization map, and whether the learned values correspond to the approximate model constructed in Section \ref{finite_app_sec}.
 We present quantized Q-learning algorithms and show that they indeed converge to the optimal Q-values of the approximate models constructed in Section \ref{finite_app_sec} for a particular weighting measure that is given by the invariant measure of the state process.} When one runs the $Q$-learning algorithm, it is important to note that the quantized process is not an MDP, and in fact should be viewed as a POMDP, a viewpoint used by \cite{kara2021convergence} and by \cite{KSYContQLearning}. For the synchronous algorithm, we use the properties of the span semi-norm. For the asynchronous setup, we generalize the proof method given by \cite{kara2021convergence} for the average cost criterion under certain ergodicity properties induced by an exploration policy, though with additional technical analysis as the $Q$-iterates do not satisfy the boundedness properties a priori unlike the discounted cost criterion setup. In particular, in Section \ref{syn_q_sec} we present and study a synchronous Q-learning algorithm, and in Section \ref{asyn_q_sec}, we present and study an asynchronous algorithm. Theorem \ref{syn_th} establishes the convergence of a synchronous Q-learning algorithm. Theorem \ref{quant_q} shows the convergence of an asynchronous Q-learning algorithm. 


\item[(iii)] [Convergence to Near-Optimality for Continuous State/Action Models] For both the synchronous and the asynchronous quantized Q-learning algorithms, the limit is shown to be the fixed-point solution of the optimality equation of an approximate model as in (i) above, and thus the convergence is to near-optimal policies. 

\end{itemize}

\sy{To give a taste of the main results, we summarize our approximation results for finite action spaces. Approximation of continuous action spaces by finite sets is discussed in Section~\ref{finiteActSp}; accordingly, we assume throughout this summary that the action space is finite.

We consider a finite partition of the state space $\mathds{X}$, given by $\{B_i\}_{i=1}^M$. We denote by $L_\mathds{X}$  the worst case distortion over the partition bins under a given weight measure (see \eqref{loss_func} for the definition).  A precise statement of the following theorem is  given in Theorem~\ref{main_thm}.

\begin{mainresult}
Suppose $\hat{\pi}$  is optimal for the approximate model. For a given initial state $x_0\in\mathds{X}$, consider the average-cost of $\hat{\pi}$ in the original MDP denoted by $J(x_0,\hat{\pi})$, and the optimal average-cost value of the original MDP denoted by $J^*(x_0)$. Furthermore, the optimal average cost value of the approximate MDP defined in \eqref{lifted_model}-\eqref{finite_cost} is  denoted by $\hat{J}(x_0)$ for initial state $x_0\in\mathds{X}$.  

Under Assumption \ref{mix_kernel}, these values are independent of the initial state and we write
\begin{align*}
 \rho(\hat{\pi})=J(x_0,\hat{\pi}),\quad \hat{\rho}=\hat{J}(x_0),\quad  \rho=J^*(x_0)  
 \end{align*}
 for all $x_0\in\mathds{X}$.
Here, $K_c$ denotes the Lipschitz constant of the one-stage cost function and
$K_{\mathcal T}$ denotes the Lipschitz (contraction) constant of the transition kernel,
as defined in Assumption~\ref{lip_assmp}.

Fix $\delta$ such that $K_\mathcal{T}<\left(\frac{1}{\kappa}\right)^{\frac{1-\delta}{\delta}}$ where $\kappa:=1-\nu(\mathds{X})>0$, and let $D$ denote the diameter of the state space.  
\begin{itemize}
\item[i.] Under  Assumption \ref{lip_assmp} with $K_\mathcal{T}<1$, we have that
\begin{align*}
|\rho-\hat{J}(x_0)|\leq  \frac{K_c}{1-K_\mathcal{T}}L_\mathds{X}.
\end{align*}
\item[ii.] Under  Assumption \ref{mix_kernel} and Assumption \ref{lip_assmp}, 
\begin{align*}
\left|{\rho}-\hat{\rho}\right|\leq\frac{K_cD^{1-\delta}}{1-K_\mathcal{T}^\delta\kappa^{1-\delta}}\left(L_\mathds{X}\right)^\delta
\end{align*}
Furthermore, 
\begin{align*}
\rho(\hat{\pi})-\rho\leq \frac{2K_cD^{1-\delta}}{(1-\kappa)\left(1-K_\mathcal{T}^\delta\kappa^{1-\delta}\right)}\left(L_\mathds{X}\right)^\delta
\end{align*}
\item[iii.] If the quantization is such that $L_\mathds{X}=\frac{1}{n}$ (which is possible as $\mathds{X}$ is assumed to be compact), then by denoting the learned policy by $\pi_n$, under Assumptions \ref{mix_kernel} and  \ref{weak:as1}, we have that \[\lim_{n\rightarrow\infty} \rho(\pi_n) =\rho.\]
\item [iv.] Running  Q-learning on quantized states (called Quantized Q-Learning with a synchronous Algorithm \ref{Qit1} and asynchronous Algorithm \ref{Qit_async}) converges and results in a policy $\hat{\pi}$. This policy corresponds to an optimal policy for a discretized approximate model, with discretization bins weighted by the stationary distribution of the state process induced by an exploration policy. As a corollary of the above, the convergence is to near optimality.
\end{itemize}
\end{mainresult}
}
\section{Average Cost Optimality Equation and Contraction Properties of Relative Value Iteration}

We start our analysis by first reviewing technical tools needed throughout the paper and related results on average cost optimality.

\subsection{Convergence Notions for Probability Measures and Regularity Properties of Transition Kernels}
For the analysis of the technical results, we use different notions of convergence for sequences of probability measures.

Two important notions of convergence for sequences of probability measures are weak convergence and convergence under total variation. A sequence $\{\nu_n,n\in\N\}$ in $\mathcal{P}(\mathds{X})$ is said to converge to $\nu\in\mathcal{P}(\mathds{X})$ \emph{weakly} if $\int_{\mathds{X}}c(x)\nu_n(dx) \to \int_{\mathds{X}}c(x)\nu(dx)$ for every continuous and bounded $c:\mathds{X} \to \R$.

  For probability measures $\mu,\nu \in \mathcal{P}(\mathds{X})$, the \emph{total variation} metric is given by
  \begin{align*}
    \|\mu-\nu\|_{TV}&=2\sup_{B\in\mathcal{B}(\mathds{X})}|\mu(B)-\nu(B)|=\sup_{f:\|f\|_\infty \leq 1}\left|\int f(x)\mu(\dd x)-\int f(x)\nu(\dd x)\right|,
  \end{align*}
  \noindent where the supremum is taken over all measurable real-valued functions $f$ such that $\|f\|_\infty=\sup_{x\in\mathds{X}}|f(x)|\leq 1$. A sequence $\nu_n$ is said to converge in total variation to $\nu \in \mathcal{P}(\mathds{X})$ if $\|\nu_n-\nu\|_{TV}\to 0$.

 {Finally, for probability measures $\mu,\nu \in \mathcal{P}(\mathds{X})$ with finite first-order moments (that is, $\int \|x\| \, d\mu$ and $\int \|x\| \, d\nu $ are finite)}, the \emph{first order Wasserstein} distance is defined as 
\begin{align*}
W_1(\mu,\nu)=\inf_{\pi(\mu,\nu)}E[d_\mathds{X}(X,Y)]=\sup_{f: Lip(f)\leq 1}|\int f(x)\mu(dx)-\int f(x)\nu(dx)|
\end{align*}
where $\pi(\mu,\nu)$ denotes all possible couplings of $X$ and $Y$ with marginals $X\sim\mu$ and $Y\sim\nu$, {and 
\begin{align}
Lip(f) := \sup_{e \neq e'} \frac{\left|f(e) - f(e')\right|}{d_\mathds{X}(e,e')},\nonumber
\end{align}}
and the second {equality follows from the dual formulation of the Wasserstein distance \cite[Remark 6.5]{Vil09}}. Note that weak convergence and Wasserstein convergence are equivalent if the underlying space is compact.

We now introduce the class of H\"older continuous functions and consider their use as test functions for studying probability measures.
 Let $f : \mathds{X} \to \mathbb{R}$.
The function $f$ is said to be {H\"older continuous of order}
$\delta \in (0,1]$ if there exists a constant $C>0$ such that
\[
|f(x) - f(y)| \le C\, d_\mathds{X}(x,y)^{\delta},
\qquad \forall x,y \in \mathds{X}.
\]
Here, $\delta$ is called the \emph{H\"older exponent} and $C$ the
\emph{H\"older constant}. We denote the optimal {H\"older constant} of $f$ by $[f]_{H^\delta}$ such that
\begin{align}\label{hold_def}
[f]_{H^\delta} := \sup_{x \neq y} \frac{|f(x)-f(y)|}{d_\mathds{X}(x,y)^\delta}.
\end{align}
\begin{lemma}{\citep[see][Proposition 10]{maran2023tight}}\label{hold_ineq_lem}
For any $\delta$-H\"older $f$, we have 
\begin{align*}
\left| \int f(x)P(dx)-\int f(x)P'(dx)\right|\leq [f]_{H^\delta} W_1(P,P')^\delta.
\end{align*}
\end{lemma}

We define the following regularity properties for the transition kernels:
\begin{itemize}
\item $\mathcal{T}(\cdot|x,u)$ is said to be weakly continuous in $(x,u)$ (or weak Feller), if $\mathcal{T}(\cdot|x_n,u_n)\to \mathcal{T}(\cdot|x,u)$ weakly for any $(x_n,u_n)\to (x,u)$.
\item $\mathcal{T}(\cdot|x,u)$ is said to be continuous under total variation in  $(x,u)$, if $\|\mathcal{T}(\cdot|x_n,u_n)- \mathcal{T}(\cdot|x,u)\|_{TV}\to 0$  for any $(x_n,u_n)\to (x,u)$.
\item $\mathcal{T}(\cdot|x,u)$ is said to be continuous under the first-order Wasserstein distance in $(x,u)$, if \[W_1(\mathcal{T}(\cdot|x_n,u_n), \mathcal{T}(\cdot|x,u))\to 0\]  for any $(x_n,u_n)\to (x,u)$. {To ensure continuity of $\mathcal{T}$ with respect to the first-order Wasserstein distance, in addition to weak continuity, we may assume that there exists a function $g:[0,\infty) \rightarrow [0,\infty)$ such that $\frac{g(t)}{t} \uparrow \infty$ as $t \to \infty$, and 
$$\sup_{(x,u) \in K \times \mathds{U}} \int g(\|y\|) \, \mathcal{T}(dy|x,u) < \infty$$ for any compact $K \subset \mathds{X}$. Note that the latter condition implies uniform integrability of the collection of random variables with probability measures ${\cal T}(dx_1|X_0=x_n,U_0=u_n)$ as $(x_n, u_n) \to (x,u)$, which, coupled with weak convergence, can be shown to imply convergence under the Wasserstein distance.}
\end{itemize}
\begin{example}\label{examples}
Some example models satisfying these regularity properties are as follows:
\begin{itemize}
\item[(i)] For a model with the dynamics $x_{t+1}=f(x_t,u_t,w_t)$, the induced transition kernel $\mathcal{T}(\cdot|x,u)$ is weakly continuous in $(x,u)$ if $f(x,u,w)$ is a continuous function of $(x,u)$, since for any continuous and bounded function $g$
\begin{align*}
&\int g(x_1)\mathcal{T}(dx_1|x_n,u_n)=\int g(f(x_n,u_n,w))\nu(dw)\\
&\to\int g(f(x,u,w))\nu(dw)=\int g(x_1)\mathcal{T}(dx_1|x,u)
\end{align*}
where $\nu$ denotes the probability measure of the noise process.
{If we also have that $\mathds{X}$ is compact, the transition kernel $\mathcal{T}(\cdot|x,u)$ is also continuous under the first order Wasserstein distance}.

\item[(ii)] For a model with the dynamics $x_{t+1}=f(x_t,u_t)+w_t$, the induced transition kernel $\mathcal{T}(\cdot|x,u)$ is continuous under total variation in $(x,u)$ if $f(x,u)$ is a continuous function of $(x,u)$, and $w_t$ admits a continuous density function. 
\item[(iii)] In general, if the transition kernel admits a continuous density function $f$ such that $\mathcal{T}(dx_1|x,u)=f(x_1,x,u)\lambda(dx_1)$, then $\mathcal{T}(dx_1|x,u)$ is continuous in total variation. This follows from an application of  Scheff\'e's Lemma \cite[Theorem 16.12]{Bil95}. In particular, we can write that
\begin{align*}
\|\mathcal{T}(\cdot|x_n,u_n)-\mathcal{T}(\cdot|x,u)\|_{TV}=\int_{\mathds{X}}|f(x_1,x_n,u_n)-f(x_1,x,u)|\lambda(dx_1)\to 0.
\end{align*}
\item[(iv)] For a model with the dynamics $x_{t+1}=f(x_t,u_t,w_t)$, if $f$ is Lipschitz continuous in the $(x,u)$ pair, that is, there exists some $\alpha<\infty$ such that
\begin{align}\label{LipDepW}
d_\mathds{X}\left(f(x_n,u_n,w),f(x,u,w)\right)\leq \alpha\left(d_\mathds{X}(x_n,x)+d_\mathds{U}(u_n,u)\right),
\end{align}
we can then bound the first order Wasserstein distance between the corresponding kernels by $\alpha$:
\begin{align}
&W_1\left(\mathcal{T}(\cdot|x_n,u_n),\mathcal{T}(\cdot|x,u)\right)=\sup_{Lip(g)\leq 1}\left|\int g(x_1)\mathcal{T}(dx_1|x_n,u_n)-\int g(x_1)\mathcal{T}(dx_1|x,u) \right|\nonumber \\
&=\sup_{Lip(g)\leq 1}\left|\int g(f(x_n,u_n,w))\nu(dw)-\int g(f(x,u,w))\nu(dw)\right| \nonumber \\
&\leq \int d_\mathds{X}\left(f(x_n,u_n,w),f(x,u,w)\right)\nu(dw)\leq \alpha\left(d_\mathds{X}(x_n,x)+d_\mathds{U}(u_n,u)\right).\label{averageLipB}
\end{align}
Observe that the Lipschitz bound in (\ref{LipDepW}) may exhibit dependency on the realization of $w$, whose average growth under measure $\nu$ can lead to a more relaxed bound for the transition kernel deviation in (\ref{averageLipB}).
\end{itemize}
\end{example}

\subsection{The Average Cost Optimality Equation}\label{acoe_section}

Consider the following average cost problem:
\begin{eqnarray}\label{AverageCostProblemDef2}
J^*(x):= \inf_{\pi \in \Pi_A} J(x,\pi) = \inf_{\pi \in \Pi_A} \limsup_{T \to \infty} {1 \over T} E^{\pi}_x [\sum_{t=0}^{T-1} c(x_t,u_t)].
\end{eqnarray}
To study the average cost problem, one common approach is to establish the existence of an average cost optimality equation (ACOE), and an associated verification theorem. 

\begin{definition}\label{triplet}
The collection of measurable functions $\rho: \mathbb{X} \to \mathbb{R}, h : \mathbb{X} \to \mathbb{R}, f: \mathbb{X} \to \mathbb{U}$ is a canonical triplet if for all $x \in \mathbb{X}$, 
\[\rho(x)=\inf_{u \in \mathbb{U}} \int \rho(x') {\cal T}(dx'|x,u)\]
\[\rho(x) + h(x) = \inf_{u \in \mathbb{U}} \bigg(c(x,u) + \int h(x') {\cal T}(dx'|x,u) \bigg) \]
with
\[\rho(x) = \int \rho(x') {\cal T}(dx'|x,f(x))\]
\[\rho(x) + h(x) = \bigg(c(x,f(x)) + \int h(x') {\cal T}(dx'|x,f(x)) \bigg) \]
We will refer to these relations as the average cost optimality equation (ACOE).
\end{definition}
The following verification theorem is a standard result \citep[see][]{survey,HernandezLermaMCP}
\begin{proposition}\label{acoe_prop}
Let $\rho,h,f$ be a canonical triplet.
If $\rho$ is a constant and 
\begin{eqnarray}
\limsup_{n \to \infty} {1 \over n}E^{\pi}_x[h(x_n)] =0, \label{condConv01}
\end{eqnarray}
 for all $x$ and under every policy $\pi$, then the stationary deterministic policy $\pi^* = \{f,f,f,\cdots\}$ is optimal, such that
\[\rho= J(x,\pi^*) = \inf_{\pi \in \Pi_A} J(x,\pi)\]
where
\[J(x,\pi) = \limsup_{T \to \infty} {1 \over T} E^{\pi}_x[\sum_{k=0}^{T-1} c(x_t,u_t)].\]
\end{proposition}

In the following, we present  two different approaches to establish existence of solutions to the average cost optimality problem. One approach is via a direct contraction argument, which relates average cost optimality to discounted cost optimality of an equivalent problem; and another using a Wasserstein contraction argument. These approaches will also be used to establish error bounds for the approximation methods presented in the paper.

\subsubsection{Contraction under the sup norm by equivalence with a discounted cost problem}\label{secMinorizationACOE}
We have the following minorization condition.

\begin{assumption}\label{mix_kernel}
There exists a positive measure $\nu$ such that
\[{\cal T}(B | x,u) \geq \nu(B),\]
for all $B \in {\cal B}(\mathbb{X})$ and for all $(x,u)\in\mathds{X}\times\mathds{U}$.
\end{assumption}
Under Assumption \ref{mix_kernel}, define ${\cal T}'(\cdot|x,u) = {\cal T}(\cdot|x,u) - \nu(\cdot)$, which is a positive measure. Then, the map
\begin{align}\label{contractAverageCostArgum}
(\mathbb{T}'(f))(x) = \min_{u\in\mathds{U}} \bigg(c(x,u) + \int f(x_1){\cal T}'(dx_1 | x,u)\bigg)
\end{align}
is a contraction with contraction constant $\kappa:=1-\nu(\mathds{X})<1$ (see \cite[p.61]{hernandez2012adaptive} for a historical review on this approach). With this assumption, one can apply the standard value iteration algorithm using $\mathbb{T}'$. The limit equation
\begin{align}
f(x) &= \min_{u\in\mathds{U}} (c(x,u) + \int f(x_1){\cal T}'(dx_1 | x,u)) \nonumber \\
&= \min_{u\in\mathds{U}}  (c(x,u) + \int f(x_1){\cal T}(dx_1 | x,u)) - \int f(x_1) \nu(dx_1)
\end{align}
is the desired ACOE in Definition \ref{triplet} with $\rho \equiv \int f(x_1) \nu(dx_1)$. The existence of a minimizing control policy is ensured by measurable selection conditions, assuming either weak continuity of the kernel (in both the state and action), or strong continuity (of the kernel in actions for every state) properties. We consider the former in the following:
\begin{assumption}
\label{weak:as}
\begin{itemize}
\item [(a)] The one-stage cost function $c$ is bounded and continuous.
\item [(b)] The stochastic kernel ${\cal T}(\,\cdot\,|x,u)$ is weakly continuous in $(x,u) \in {\mathbb X} \times \mathbb{U}$ (that is, weak Feller).
\item [(c)] $\mathbb{U}$ is compact.
\item [(d)] $\mathbb{X}$ is compact.
\end{itemize}
\end{assumption}

The corresponding measurable selection criteria are given by \citet[Theorem 2]{himmelberg1976optimal}, \cite{Schal}, \cite{schal1974selection} and \cite{kuratowski1965general}. We also refer the reader to the book by \cite{HernandezLermaMCP} for a comprehensive analysis and detailed literature review. One can then show \citep[see e.g.][] {survey,HernandezLermaMCP,hernandezlasserre1999further,GoHe95,Veg03,demirci2023average,yuksel2020control} that
under Assumptions \ref{mix_kernel} and \ref{weak:as}  there exists a solution to the average cost optimality equation.

\subsubsection{Contraction under the Wasserstein-1 distance}

The following is the main regularity assumption of the paper. 
\begin{assumption}\label{lip_assmp2}
\begin{itemize}
\item The original cost function $c$ is Lipschitz, such that $|c(x,u)-c(x',u')|\leq K_c (d_{\mathbb{X}}(x,x')+d_{\mathbb{U}}(u,u'))$ for some $K_c<\infty$ for all $x,x',u,u'$.
\item The transition kernel $\mathcal{T}$ is Lipschitz continuous under the first order Wasserstein distance such that $W_1(\mathcal{T}(\cdot|x,u),\mathcal{T}(\cdot|x',u'))\leq K_\mathcal{T} (d_{\mathbb{X}}(x,x')+d_{\mathbb{U}}(u,u'))$ for some $K_\mathcal{T}<\infty$ for all $x,x',u,u'$.
\item $\mathds{X}$ and $\mathds{U}$ are compact.
\end{itemize}
\end{assumption}

Under Assumption \ref{lip_assmp2} with $K_\mathcal{T}<1$, \citet{demirci2023average} showed that the ACOE for the original model admits a solution, and that 
\begin{align*}
\lim_{\beta\to 1}(1-\beta) J_\beta^*(x) = \rho.\nonumber
\end{align*}
where $J_\beta^*(x)$ denotes the optimal value function under the discounted cost criteria with  discount factor $0<\beta<1$ (see (\ref{DiscCostProblemDef})).




\section{Near Optimality of Quantized State and Action Space Approximations}\label{finite_app_sec}
In the following, we build on the work of \cite{SaYuLi15c,SaLiYuSpringer} to construct approximate MDPs with finite state and action spaces. Unlike \cite{SaYuLi15c,SaLiYuSpringer}, we require weaker conditions for the average cost criterion: notably, only weak convergence is sufficient for the approximation results, rather than the total variation continuity assumed by \cite{SaYuLi15c,SaLiYuSpringer}.

\subsection{Finite Action Approximate MDP: Quantization of the Action Space}\label{finiteActSp}
Under Assumption \ref{weak:as}, the action space $\mathbb{U}$ is compact, and hence totally bounded. Therefore, one can find a finite set $\Lambda = \{u_{1},\ldots,u_{k}\} \subset \mathbb{U}$ such that 
\begin{align}\label{act_error}
\sup_{u\in\mathds{U}}\min_{\hat{u}\in\Lambda} d_{\mathbb{U}}(u,\hat{u}) =:\E(\Lambda)<\infty.
\end{align}

Consider the MDP with tuple $(\mathds{X}, \Lambda,\mathcal{T},c)$. Note that under Assumption \ref{mix_kernel}, both the original MDP and the finite action space MDP admit  solutions to the ACOE. Furthermore, under Assumption \ref{weak:as} or Assumption \ref{lip_assmp2}, there exist  optimal policies for the MDPs under both continuous and the finite action spaces.   We denote the optimal values under the average cost criteria by $\rho_\mathds{U}$ and $\rho_\Lambda$, respectively. 

For the MDPs with the continuous and finite action spaces, we write the following ACOE's:
\begin{align*}
&h(x)=\inf_{u\in\mathds{U}}\left\{c(x,u)+\int h(x_1)\mathcal{T}(dx_1|x,u)-\int h(x_1)\nu(dx_1)\right\}\nonumber\\
&\hat{h}(x)=\min_{u\in\Lambda}\left\{{c}(x,u)+\int \hat{h}(x_1){\mathcal{T}}(dx_1|x,u)-\int \hat{h}(x_1)\nu(dx_1)\right\}.
\end{align*}
Note that $\rho_\mathds{U}=\int {h}(x_1)\nu(dx_1)$ and $\rho_\Lambda=\int \hat{h}(x_1)\nu(dx_1)$. We follow the proof strategy by \cite{maran2023tight}, and use the H\"older continuity of the relative value functions for the near-optimality analysis. 
\begin{lemma}\label{holder_lemma}
Let  $\kappa=1-\nu(\mathds{X})>0$, and let $D$ denote the diameter of the state space $\mathds{X}$. Fix $\delta>0$ such that
$K_{\mathcal T} < \left(\frac{1}{\kappa}\right)^{\frac{1-\delta}{\delta}}$. Under Assumptions  \ref{mix_kernel} and \ref{lip_assmp2} 
\begin{align*}
[h]_{H^\delta} \leq \frac{K_c D^{1-\delta}}{1-\kappa\left(\frac{K_\mathcal{T}}{\kappa}\right)^\delta}.
\end{align*}
\end{lemma}
\begin{proof}
The proof can be found in Appendix \ref{holder_lemma_proof}.
\end{proof}
Then the following holds:
\begin{theorem} \label{weak:mainthm2}
Let  $\kappa=1-\nu(\mathds{X})>0$, and let $D$ denote the diameter of the state space $\mathds{X}$. Fix $\delta>0$ such that
$K_{\mathcal T} < \left(\frac{1}{\kappa}\right)^{\frac{1-\delta}{\delta}}$.
Then the following hold:
\begin{itemize}
\item[i.] Under Assumptions  \ref{mix_kernel} and \ref{weak:as}, we have  $\rho_\Lambda\to \rho_\mathds{U}$ as $\E(\Lambda)\to 0$.
\item[ii.] Under Assumption \ref{lip_assmp2}, if $K_\mathcal{T}<1$, then  
\begin{align*}
\rho_\Lambda -\rho_\mathds{U}\leq \frac{K_c}{1-K_\mathcal{T}} \E(\Lambda).
\end{align*}
\item[iii.]  Under Assumptions  \ref{mix_kernel} and \ref{lip_assmp2}, we have  
\begin{align*}
\rho_\Lambda -\rho_\mathds{U}\leq  \frac{K_cD^{1-\delta}}{1-\kappa\left(\frac{K_\mathcal{T}}{\kappa}\right)^\delta}\E(\Lambda)^\delta.
\end{align*}
\end{itemize}
\end{theorem}

\begin{proof}
(i) follows from \cite[]{saldi2014near},\cite[Theorem 3.22]{SaLiYuSpringer}. For (ii) under Assumption \ref{lip_assmp2} with $K_\mathcal{T} <1$, one can show that (see \cite{demirci2023average})
\begin{align*}
\lim_{\beta\to 1} (1-\beta)J_\beta^*(x) = \rho_\mathds{U}\\
\lim_{\beta\to 1} (1-\beta)\hat{J}_\beta(x) = \rho_\Lambda
\end{align*}
for any $x\in\mathds{X}$ where $J_\beta^*$  (respectively $\hat{J}_\beta$) represents the optimal discounted value function under the original action space $\mathds{U}$ (respectively under the action space $\Lambda$). Furthermore, we also have the following upper-bound for the discounted value function difference (see e.g. \cite{SaLiYuSpringer}):
\begin{align*}
\left|\hat{J}_\beta(x)-J_\beta^*(x)\right| \leq \frac{K_c}{(1-\beta)(1-\beta K_\mathcal{T})} \E(\Lambda).
\end{align*}
Hence, combining these two bounds gives the desired bound. Finally, for (iii), we work with the following ACOE's
\begin{align*}
&h(x)=\inf_{u\in\mathds{U}}\left\{c(x,u)+\int h(x_1)\mathcal{T}(dx_1|x,u)-\int h(x_1)\nu(dx_1)\right\}\nonumber\\
&\hat{h}(x)=\min_{u\in\Lambda}\left\{{c}(x,u)+\int \hat{h}(x_1){\mathcal{T}}(dx_1|x,u)-\int \hat{h}(x_1)\nu(dx_1)\right\}.
\end{align*}
Note that we have $\rho_\mathds{U}=\int h(x_1)\nu(dx_1)$ and $\rho_\Lambda=\int \hat{h}(x_1)\nu(dx_1)$. Thus, it suffices to find an upper bound on $\|h-\hat{h}\|_\infty$. Let $\hat{u}^*$ denote the minimizer of the second equation. Furthermore, let $\hat{u}=\argmin_{\Lambda} d_{\mathds{U}}(u^*,u)$ denote the closest element in the finite action set to the minimizer $u^*$ of the first equation whose existence is guaranteed under Assumption \ref{lip_assmp2}.

 We define $R(\cdot|x,u):=\frac{\mathcal{T}(\cdot|x,u)-\nu(\cdot)}{\kappa}$, with $\kappa:=1-\nu(\mathds{X})$ which is a  Markov kernel so that $R(\mathds{X}|x,u)=1$. It is then easy to check that 
\begin{align*}
W_1(R(\cdot|x,u^*),{R}(\cdot|x,\hat{u}))\leq \frac{K_\mathcal{T}}{\kappa} \E(\Lambda)
\end{align*}
using that fact that $d_{\mathds{U}}({u^*},\hat{u})\leq \E(\Lambda)$ by construction. One can write the above ACOEs as
\begin{align*}
&h(x)=c(x,u^*)+\kappa\int h(x_1)R(dx_1|x,u^*)\nonumber\\
&\hat{h}(x)={c}(x,\hat{u}^*)+\kappa\int \hat{h}(x_1){R}(dx_1|x,\hat{u}^*)
\end{align*}
We note that $\hat{h}\geq h$. Since $\hat{u}$ is not the optimal action for the finite action space in general, we can then write:
\begin{align*}
\hat{h}(x)-h(x) & \leq \left|c(x,\hat{u})-c(x,u^*)\right| + \kappa  \left|  \int \hat{h}(x_1)R(dx_1|x,\hat{u}) - \int h(x_1)R(dx_1|x,u^*)\right|\\
&\leq {K_c}\E(\Lambda) + \kappa   \left|  \int \hat{h}(x_1)R(dx_1|x,\hat{u}) - \int h(x_1)R(dx_1|x,\hat{u})\right|\\
&\qquad\qquad +  \kappa   \left|  \int h(x_1)R(dx_1|x,\hat{u}) - \int h(x_1)R(dx_1|x,u^*)\right|\\
&\leq {K_c}\E(\Lambda) +\kappa \|h-\hat{h}\|_\infty +  \kappa[h^*]_{H^\delta} \left(\frac{K_\mathcal{T}}{\kappa}\E(\Lambda)\right)^\delta
\end{align*}
where $[h]_{H^\delta}$ denotes the H\"older continuity constant of $h$ under $\delta$. Above, we also used Lemma \ref{hold_ineq_lem}.

Writing ${K_c}\E(\Lambda)  = K_c \E(\Lambda)^\delta\E(\Lambda)^{1-\delta}$, we can rearrange the terms above to get
\begin{align*}
\|h-\hat{h}\|_\infty \leq \frac{K_c \E(\Lambda)^{1-\delta} + [h]_{H^\delta} \kappa\left(\frac{K_\mathcal{T}}{\kappa}\right)^\delta}{1-\kappa} \E(\Lambda)^\delta.
\end{align*}
Using Lemma \ref{holder_lemma} for $[h]_{H^\delta}$, we write
\begin{align}\label{h_dif_act}
\|h-\hat{h}\|_\infty \leq \frac{K_cD^{1-\delta}}{\left(1-\kappa\left(\frac{K_\mathcal{T}}{\kappa}\right)^\delta\right)(1-\kappa)}\E(\Lambda)^\delta.
\end{align}
The proof is concluded by noting that $\rho_\Lambda - \rho_\mathds{U} \leq \|\hat{h}-{h}\|_\infty \nu(\mathds{X})$.

\end{proof}
In particular, $\rho_\Lambda$ denotes the average cost of the best policy that uses actions from the finite set $\Lambda$, but evaluated under the dynamics of the original MDP.
In view of the above, for the rest of the paper, we will assume that the action space $\mathds{U}$ is finite. Error bounds for a continuous action space problem can be obtained by appropriately discretizing the action space and using Theorem \ref{weak:mainthm2}. As such, we will replace Assumption \ref{lip_assmp2} with the following:
\begin{assumption}\label{lip_assmp}
\begin{itemize}
\item $\mathds{X}$ is compact and $\mathds{U}$ is finite.
\item The original cost function $c$ is Lipschitz, such that $|c(x,u)-c(x',u)|\leq K_c d_{\mathbb{X}}(x,x')$ with constant $K_c<\infty$ for all $x,x',u$.
\item The transition kernel $\mathcal{T}$ is Lipschitz continuous under the first order Wasserstein distance such that $W_1(\mathcal{T}(\cdot|x,u),\mathcal{T}(\cdot|x',u))\leq K_\mathcal{T} d_{\mathbb{X}}(x,x')$ for some $K_\mathcal{T}<\infty$ for all $x,x',u$.
\end{itemize}
\end{assumption}

\subsection{MDP Approximation Results}\label{finite_state_sec}
We construct an approximate MDP via state space discretization, following the construction of \cite{SaYuLi15c,SaLiYuSpringer}. We choose a collection of disjoint sets $\{B_i\}_{i=1}^M$ such that $\bigcup_i B_i=\mathds{X}$, and $B_i\bigcap B_j =\emptyset$ for any $i\neq j$. 
We choose a representative state, $y_i\in B_i$, for each disjoint set. For this setting, we denote the new finite state space by
$\mathds{Y}:=\{y_1,\dots,y_M\}$, and the mapping from the original state space $\mathds{X}$ to the finite set $\mathds{Y}$ is done via
\begin{align}\label{quant_map}
q(x)=y_i \quad \text{ if } x\in B_i.
\end{align}

Furthermore, we choose a weight measure $\mu\in\P(\mathds{X})$ on $\mathds{X}$ such that $\mu(B_i)>0$ for all $B_i$. We now define normalized measures using the weight measure on each separate quantization bin $B_i$ as follows:  
\begin{align}\label{norm_inv}
\mu_{y_i}(A):=\frac{\mu(A)}{\mu(B_i)}, \quad \forall \text{ }\mathrm{ Borel} \text{ }A \subset B_i, \quad \forall i\in \{1,\dots,M\},
\end{align}
that is, $\mu_{y_i}$ is the normalized weight measure on the set $B_i$, where $y_i$ belongs to. 
We then define  the cost  function and transition kernel for the approximate model as follows:
\begin{align}\label{lifted_model}
\hat{c}(x,u)&=\int_{B_i}c(z,u)  \mu_{y_i}(dz)\nonumber\\
\hat{\mathcal{T}}(A|x,u)&= \int_{B_i}\mathcal{T}(A|z,u)\, \mu_{y_i}(dz)
\end{align}
for any  $A\in\B(\mathds{X})$  and any  $x\in B_i$. We denote the optimal average cost for this approximate model by $\hat{J}(x)$ given the initial condition is $x$.

Note that although the approximate model is defined over the original state space $\mathds{X}$, by construction its cost function and transition kernel are constant over the quantization bins $\{B_i\}_{i=1}^M$. Hence, the approximate model can equivalently be viewed as a finite-state MDP with state space $\mathds{Y}=\{y_1,\dots,y_M\}$ where $y_i$ is the representative state of the quantization bin $B_i$.  Indeed, for any $y_i,y_j\in \mathds{Y}$ and $u \in \mathds{U}$, the stage-wise cost and the transition kernel for the finite-state model are defined as 
\begin{align}\label{finite_cost}
C(y_i,u) &:= \int_{B_i}c(x,u) \, \mu_{y_i}(dx),\nonumber\\
P(y_j|y_i,u) &:= \int_{B_i}\mathcal{T}(B_j|x,u)\, \mu_{y_i}(dx).
\end{align}
In the following, we establish approximation and performance results for control policies obtained via the approximate MDPs when applied to the true model.

\subsubsection{MDP Approximation via Wasserstein Continuity with Modulus of Continuity in Approximation}\label{FinMDPWassM}

The main goal of this section is to establish the near-optimality of policies designed for a finite MDP when applied to the continuous model. Denoting the optimal policy of the finite model by $\hat{\pi}$, we aim to find upper bounds on 
\begin{align*}
J(x,\hat{\pi}) - J^*(x)
\end{align*}
where $J^*(x)$ is the optimal value under the ergodic cost criterion for the original model. To analyze this term, we add and subtract the optimal value of the approximate model, i.e. $\hat{J}(x)$ yielding 
\begin{align}\label{add_sub}
\left(J(x,\hat{\pi}) - \hat{J}(x)\right) + \left(\hat{J}(x)- J^*(x)\right).
\end{align}
The second term is the difference between the value functions of the approximate and the original MDPs, and we first bound this difference as an intermediate step. The first term (which we may refer to as the robustness error term) measures the difference in the expected cost between the approximate and original MDP models under the \emph{same} policy, namely, the optimal policy of the discretized model. Our goal is to derive error bounds in terms of the Wasserstein–1 Lipschitz continuity of the original transition kernel using the approximate MDP. Therefore, unless the optimal policy is Lipschitz continuous, it is generally impossible to bound this term uniformly over policies. Restricting policies to be Lipschitz a priori has been previously considered in the literature \citep[see e.g.][]{pirotta2015policy,asadi2018lipschitz,Rachelson2010OnTL,maran2023tight}. Nonetheless, we will show that when the policy is an optimal policy of the approximate MDP, it is indeed possible to derive an error bound for the second term that vanishes as the discretization becomes finer (the principal reason being that the state and corresponding optimal action pairs converge to one another due to optimality which, together with continuous weak convergence of kernels, leads to the vanishing error). We note that, as is often the case for finite space MDPs, if the original model is instead approximated in total variation distance, the value difference under the same policy can be directly bounded for any policy due to the strong regularity imposed by total variation (see \citet[Section 4.5]{kara2020robustness} and \cite{NanJiangNote2022}). However, for general space MDPs, such a direct condition is too demanding.

We define the following loss functions induced by the quantization on the cost function and the transition kernel:
\begin{align}\label{loss_func_cost}
&L_c(x,u):=\left| c(x,u) - \hat{c}(x,u)\right|\nonumber\\
&L_T(x,u):=W_1\left(\mathcal{T}(\cdot|x,u),\hat{\mathcal{T}}(\cdot|x,u)\right)
\end{align}
where $\hat{c}$ and $\hat{\mathcal{T}}$ are the cost function and transition kernel of the approximate model defined in (\ref{lifted_model}).

\noindent{\bf Value Function Difference.} For the value difference bound, that is the second term in (\ref{add_sub}), we give two alternatives: one assuming the minorization condition Assumption \ref{mix_kernel} and one without minorization but using Wasserstein contraction of the kernel with $K_\mathcal{T}<1$ under Assumption \ref{lip_assmp}.
\begin{proposition}\label{val_dif2}
Consider the optimal average-cost values of the original and the approximate models, denoted by $J^*(x)$ and $\hat{J}(x)$, respectively, for a given initial state $x$. Under  Assumption \ref{mix_kernel} and Assumption \ref{lip_assmp}, these values are constant with $\rho=J^*(x)$ and $\hat{\rho}=\hat{J}(x)$ for all $x\in \mathds{X}$.  Fix $\delta$ such that $K_\mathcal{T}<\left(\frac{1}{\kappa}\right)^{\frac{1-\delta}{\delta}}$ where $\kappa:=1-\nu(\mathds{X})>0$, and let $D$ denote the diameter of the state space. Then, we have
\begin{align*}
&|\rho-\hat{\rho}|\leq  \|L_c\|_\infty +\left(\frac{K_c D^{1-\delta}  \kappa^{1-\delta}}{1-\kappa\left(\frac{K_\mathcal{T}}{\kappa}\right)^\delta}\right) \left(\|L_T\|_\infty\right)^\delta.
\end{align*}
\end{proposition}

\begin{proof}
The proof uses a similar technique to that of Theorem \ref{weak:mainthm2}.
We will work with the following ACOE's
\begin{align*}
&h(x)=\min_u\left\{c(x,u)+\int h(x_1)\mathcal{T}(dx_1|x,u)-\int h(x_1)\nu(dx_1)\right\}\nonumber\\
&\hat{h}(x):=\min_u\left\{\hat{c}(x,u)+\int \hat{h}(x_1)\hat{\mathcal{T}}(dx_1|x,u)-\int \hat{h}(x_1)\nu(dx_1)\right\}
\end{align*}
Note that we have $\rho=\int h(x_1)\nu(dx_1)$ and $\hat{\rho}=\int \hat{h}(x_1)\nu(dx_1)$. Thus, it suffices to find an upper bound on $\|h-\hat{h}\|_\infty$. Denoting by $\kappa=1-\nu(\mathds{X})$, one can write the above as
\begin{align*}
&h(x)=\min_u\left\{c(x,u)+\kappa\int h(x_1)R(dx_1|x,u)\right\}\nonumber\\
&\hat{h}(x):=\min_u\left\{\hat{c}(x,u)+\kappa\int \hat{h}(x_1)\hat{R}(dx_1|x,u)\right\}
\end{align*}
where $R(\cdot|x,u):=\frac{\mathcal{T}(\cdot|x,u)-\nu(\cdot)}{\kappa}$ is a Markov kernel so that $R(\mathds{X}|x,u)=1$. $\hat{R}$ is defined similarly. It is then easy to check by definition that for any $x,u$
\begin{align*}
W_1(R(\cdot|x,u),\hat{R}(\cdot|x,u))\leq \frac{L_T(x,u)}{\kappa}.
\end{align*}
 Using Lemma \ref{hold_ineq_lem}, we can show the following:
 \begin{lemma}\label{h_diff_lem}
 \begin{align*}
\|h-\hat{h}\|_\infty \leq \frac{1}{1-\kappa}  \left(  \|L_c\|_\infty +\left(\frac{K_c D^{1-\delta}  \kappa^{1-\delta}}{1-\kappa\left(\frac{K_\mathcal{T}}{\kappa}\right)^\delta}\right) \left(\|L_T\|_\infty\right)^\delta   \right).
\end{align*}
 \end{lemma}
 \begin{proof}
\begin{align*}
|h(x)-\hat{h}(x)| \leq& \sup_u\left|c(x,u) - \hat{c}(x,u)\right| \nonumber\\
&+ \kappa \|h-\hat{h}\|_\infty + \kappa\sup_u\left\{ \int h(x_1){R}(dx_1|x,u) - \int h(x_1)\hat{{R}}(dx_1|x,u)\right\} \nonumber\\
&\leq \|L_c\|_\infty + \kappa \|h-\hat{h}\|_\infty +  \kappa[h]_{H^\delta} \left(\frac{\|L_T\|_\infty}{\kappa}\right)^\delta
\end{align*}
where $[h]_{H^\delta}$ denotes the H\"older continuity constant of $h$ under $\delta$. 
Rearranging the terms and using the Lemma \ref{holder_lemma} to bound $[h]_{H^\delta}$ give the desired bound.
\end{proof}
We conclude the proof noting $|\rho-\hat{\rho}|\leq \|h-\hat{h}\|_\infty \nu(\mathds{X})$ and that $\nu(\mathbb{X}) = 1 - \kappa$.
\end{proof}

The following proposition shows that if the transition kernel satisfies a Wasserstein contraction, the minorization condition is not needed to bound the value difference:
\begin{proposition}\label{val_dif}
Under  Assumption \ref{lip_assmp} with $K_\mathcal{T}<1$, we have
\begin{align*}
|\rho-\hat{J}(x)|\leq   \|L_c\|_\infty + \frac{K_c}{1- K_\mathcal{T}}\|L_T\|_\infty
\end{align*}
where $\rho=J^*(x)$ is the optimal value of the original MDP, and $\hat{J}(x)$ is the optimal value of the approximate model for the initial state $x$.  
\end{proposition}
\begin{remark}
We note that Assumption \ref{lip_assmp} with $K_\mathcal{T}<1$ guarantees that the optimal value of the original model is constant, such that $J^*(x)=\rho$, for all $x\in\mathds{X}$. However, the same may not hold for the approximate model (see Example \ref{fin_not_constant}).
\end{remark}

\begin{proof}
We first note that under Assumption \ref{lip_assmp} with $K_\mathcal{T}<1$, one can show that (see \cite{demirci2023average} or Theorem 7.3.4 of \cite{yuksel2020control}) the ACOE for the original model admits a solution and that 
\begin{align}
\lim_{\beta\to 1}(1-\beta) J_\beta^*(x) = \rho.\nonumber
\end{align}
Furthermore, we also have that for any finite MDP, and in particular, for the approximate model, we also have that 
\begin{align}
\lim_{\beta\to 1}(1-\beta) \hat{J}_\beta(x) = \hat{J}(x). \nonumber
\end{align}
From Theorem 5 of \cite{KSYContQLearning}, it follows  that 
\begin{align*}
\left|J_\beta^*(x) - \hat{J}_\beta(x)\right| \leq \frac{1}{1-\beta}\left[ \|L_c\|_\infty + \frac{K_c}{1-\beta K_\mathcal{T}}\|L_T\|_\infty\right].
\end{align*}
Hence, we can conclude the proof with a triangle inequality:
\begin{align*}
|J^*(x)-\hat{J}(x)| \leq\lim_{\beta\to 1}|(1-\beta) J_\beta^*(x) -(1-\beta) \hat{J}_\beta(x)| \leq   \|L_c\|_\infty + \frac{K_c}{1- K_\mathcal{T}}\|L_T\|_\infty.
\end{align*}
\end{proof}

\noindent{\bf On Minorization vs. Wasserstein Contraction Conditions.} To bound the difference between the value functions of the original MDP and the approximate MDP, we have considered two alternative conditions:
\begin{itemize}
\item[1.] The transitions are Wasserstein contractive in the sense that  $W_1(\mathcal{T}(\cdot|x,u), \mathcal{T}(\cdot|y,u))\leq K_\mathcal{T} d_\mathds{X}(x,y)$ with $K_\mathcal{T}<1$.
\item[2.] The transitions are Lipschitz continuous and they satisfy a minorization condition, i.e., $\mathcal{T}(\cdot|x,u)\geq \nu(\cdot)$ for a non-trivial measure $\nu(\cdot)$.
\end{itemize}
The first condition implies that the dynamics are contractive on average, which is analogous to contractive behavior in deterministic systems. In particular, under a functional representation of the dynamics via stochastic realizability conditions, we can write the dynamics as
\begin{align*}
X_{t+1}=f(X_t,U_t,W_t)
\end{align*}
for some measurable $f$, and for an i.i.d. noise sequence $W_t$ distributed according to a measure $P(\cdot)$ such that the pushforward of $P(\cdot)$ under $f(x,u,\cdot)$ is the kernel $\mathcal{T}(\cdot|x,u)$. In this representation, if we assume that there exists some function $K:\mathds{W}\times \mathds{U}\to [0,\infty)$ such that for all $x,y\in\mathds{X}$, $u\in\mathds{U}$, $w\in\mathds{W}$, 
\begin{align*}
d_\mathds{X}\left(f(x,u,w),f(y,u,w)\right) \leq K(w,u)d_\mathds{X}(x,y),
\end{align*} 
we then have 
\begin{align*}
W_1(\mathcal{T}(\cdot|x,u), \mathcal{T}(\cdot|y,u)) &\leq \sup_{\|g\|_{Lip}\leq 1} \left| \int g(f(x,u,w)) - g(f(y,u,w))  \right| P(dw)\\
& \leq \int K(w,u)P(dw) d_\mathds{X}(x,y).
\end{align*}
Thus, the Wasserstein contraction assumption for the kernel holds if $ \int K(w,u)P(dw) <1$, that is if the dynamics are contractive on average for different realizations of the noise. 

The second condition indicates that the Wasserstein contraction condition can be avoided if the dynamics are mixing, i.e.,  $\mathcal{T}(\cdot|x,u)\geq \nu(\cdot)$ for a non-trivial measure $\nu(\cdot)$.

We note that the condition $K_\mathcal{T}<1$ guarantees that the average-cost optimal value function for the original model does not depend on the initial state. However, this property may no longer hold for the finite model after discretization.  Consider the following simple example: 
\begin{example}\label{fin_not_constant}
Let $\mathds{X}=[-1,1]$, and let the dynamics be uncontrolled and given by
\begin{align*}
X_{t+1} = \frac{X_t}{2}.
\end{align*}
The cost is simply equal to the state, that is, $c(x)=x$. It is easy to see that the infinite horizon average cost for this problem is given by $\rho=0$ independent of the initial condition. Furthermore, we have that
\begin{align*}
(1-\beta)J_\beta(x) =(1-\beta) \sum_{t=0}^\infty \beta^t \left(\frac{1}{2}\right)^t x =(1-\beta) \frac{1}{1-\beta/2}x \to 0
\end{align*}
independent of the initial state $x$ as $\beta\to 1$. On the other hand, if we discretize the state space by mapping $[-1,0) \to -1$ and $[0,1]\to 1$, the infinite horizon average cost becomes $\hat{\rho}(-1)= -1$ and $\hat{\rho}(1)=1$ which depends on the initial state.  
\end{example} 

For the minorization condition, we first note that if the original model kernel satisfies this condition, then the finite model kernel also satisfies the same condition:
\begin{lemma}\label{minor_fin_lemma}
If $\mathcal{T}(\cdot|x,u)\geq \nu(\cdot)$ for a non-trivial measure $\nu(\cdot)$, then the  transition kernel for the approximate model also satisfies the same condition such that $\hat{\mathcal{T}}(\cdot|x,u)\geq \nu(\cdot)$ where $\hat{\mathcal{T}}$ is defined in (\ref{lifted_model}). Consequently, for the transition probabilities of the finite model defined in (\ref{finite_cost}), there exists some $y_j$ such that $P(y_j|y_0,u_0)>0$ for all $y_0\in\mathds{Y}$ and $u_0\in\mathds{U}$.
\end{lemma}
\begin{proof}
Fix $(x,u)\in\mathds{X}\times\mathds{U}$. Take $i$ such that $x\in B_i$. Then
\begin{align*}
\hat{\mathcal{T}}(\cdot|x,u)=\int_{B_i} \mathcal{T}(\cdot|z,u)\mu_{y_i}(dz)\geq \int \nu(\cdot)\mu_{y_i}(dz)=\nu(\cdot).
\end{align*}
Since $\nu(\cdot)$ is non-trivial, that is since $\nu(\mathds{X})>0$, for any discretization scheme there exists some index $j$  such that $\nu(B_j)>0$, and thus for any $(y_0,u_0)\in\mathds{Y}\times\mathds{U}$,
\begin{align*}
P(y_j|y_0,u_0)=\hat{\mathcal{T}}(B_j|y_0,u_0) \geq \nu(B_j)>0.
\end{align*}
\end{proof}
Even when the original model does not satisfy the minorization condition, it is sometimes possible to construct a discretization scheme which does satisfy the minorization for any rate of discretization:

\begin{example}\label{minor_counter}
Let  $\mathds{X}=[-2,2]$ and consider the dynamics $x_{k+1}=\frac{x_k}{2}+w_k$ where $w_k$ is supported on $\mathds{Q}\cap [-1,1]$  where $\mathds{Q} $ is the set of rationals. For this example, we have that $K_\mathcal{T}=\frac{1}{2}$. The kernel $\mathcal{T}$, however, does not satisfy the minorization condition. Indeed, for $x,x'\in\mathds{X}$
\begin{align*}
\supp(\mathcal{T}(\cdot|x))=\frac{1}{2}x + (\mathds{Q}\cap [-1,1]).
\end{align*}
 In particular, we have that 
 \begin{align*}
 \supp(\mathcal{T}(\cdot|x))\cap \supp(\mathcal{T}(\cdot|x'))=\emptyset \text{ when } |x-x'|\notin \mathds{Q}.
 \end{align*}
 Thus, no uniform minorization measure can exist.
  On the other hand, for any finite partition of $\mathds{X}$ that has a bin which includes an open neighborhood of the point $0$, this bin is visited with positive probability from every other bin. Therefore, the transition kernel of any approximate (finite-state) model based on such  a partition satisfies a minorization condition. 
\end{example}

\noindent{\bf Near Optimality of Approximate Policies.} The previous results give us an upper bound on the difference between the optimal average-cost value functions. We now focus on the performance of policies designed for the approximate models. 
 In this setting, we note that there might be several  policies that achieve the optimal performance for the discretized model under the average cost optimality criterion, and these policies may perform differently when applied to the original problem (see \cite[Section 4]{kara2020robustness}).

Therefore, we work with policies that satisfy the ACOE. In particular, we will find performance loss bounds for a policy $\hat{\pi}$ that satisfies:
\begin{align}\label{acoe_finite}
\hat{\rho}+\hat{h}(x)&=\hat{c}(x,\hat{\pi}(x))+\int \hat{h}(x_1)\hat{\mathcal{T}}(dx_1|x,\hat{\pi}(x))\\
&=\min_{u\in\mathds{U}}\left\{\hat{c}(x,u)+\int \hat{h}(x_1)\hat{\mathcal{T}}(dx_1|x,u)\right\}.\nonumber
\end{align}
Recall from (\ref{add_sub}), that the performance of an approximate policy is related to the difference $J(x,\hat{\pi}) - \hat{J}(x)$ as well as the corresponding value-function difference, for which we have derived upper bounds. It is not immediately clear, however, how to control this term when the policy $\hat{\pi}$ may be discontinuous and the distance between the transition kernels is controlled in the $W_1$ distance only. 
\begin{theorem}\label{pol_err}
Suppose $\hat{\pi}$ satisfies (\ref{acoe_finite}) and is optimal for the approximate model. For a given initial state $x_0\in\mathds{X}$, consider the average-cost of $\hat{\pi}$ in the original MDP denoted by $J(x_0,\hat{\pi})$, and the optimal average-cost value of the original MDP denoted by $J^*(x_0)$. 

 Under Assumptions \ref{mix_kernel} and   \ref{lip_assmp}, these values are independent of the initial state with $\rho(\hat{\pi})=J(x_0,\hat{\pi})$ and $\rho=J^*(x_0)$, for all $x_0\in\mathds{X}$. Fix $\delta$ such that $K_\mathcal{T}<\left(\frac{1}{\kappa}\right)^{\frac{1-\delta}{\delta}}$ where $\kappa:=1-\nu(\mathds{X})>0$, and let $D$ denote the diameter of the state space. Then, we have
\begin{align*}
\rho(\hat{\pi})-\rho\leq   \frac{2}{1-\kappa}\left(  \|L_c\|_\infty +\left(\frac{K_c D^{1-\delta}  \kappa^{1-\delta}}{1-\kappa\left(\frac{K_\mathcal{T}}{\kappa}\right)^\delta}\right) \left(\|L_T\|_\infty\right)^\delta  \right).
\end{align*}
\end{theorem}

\begin{proof}
We start by noting that under Assumption \ref{mix_kernel} and using Lemma \ref{minor_fin_lemma}, we have 
\begin{align}\label{cont_eq}
\mathcal{T}(\cdot|x,u)\geq \nu(\cdot),\qquad \hat{\mathcal{T}}(\cdot|x,u)\geq\nu(\cdot).
\end{align}
In particular, if we define the following operators
\begin{align}\label{operators}
&Th(x):=c(x,\hat{\pi}(x))+\int h(x_1)\mathcal{T}(dx_1|x,\hat{\pi}(x))-\int h(x_1)\nu(dx_1)\nonumber\\
&\hat{T}h(x):=\hat{c}(x,\hat{\pi}(x))+\int h(x_1)\hat{\mathcal{T}}(dx_1|x,\hat{\pi}(x))-\int h(x_1)\nu(dx_1)\nonumber\\
&T^*h(x):=c(x,\pi^*(x))+\int h(x_1)\mathcal{T}(dx_1|x,\pi^*(x))-\int h(x_1)\nu(dx_1),
\end{align}
where $\pi^*$ is optimal for the average cost problem and solves the ACOE. Using (\ref{cont_eq}), one can then show (see (\ref{contractAverageCostArgum})), that these operators are contractions with contraction constant $\kappa:=1-\nu(\mathds{X})<1$. 

Furthermore, these operators admit unique fixed points, say $h(x),\hat{h}(x),h^*(x)$ respectively. As in the proof of Theorem \ref{val_dif2}, using the kernels $R(\cdot|x,u)=\frac{\mathcal{T}(\cdot|x,u)-\nu(\cdot)}{\kappa}$ and  $\hat{R}(\cdot|x,u)=\frac{\hat{\mathcal{T}}(\cdot|x,u)-\nu(\cdot)}{\kappa}$  we write these fixed point equations as:
\begin{align}\label{operators2}
&h(x)=c(x,\hat{\pi}(x))+\kappa\int h(x_1){R}(dx_1|x,\hat{\pi}(x))\nonumber\\
&\hat{h}(x)=\hat{c}(x,\hat{\pi}(x))+\kappa\int \hat{h}(x_1)\hat{{R}}(dx_1|x,\hat{\pi}(x))\nonumber\\
&h^*(x)=c(x,\pi^*(x))+\kappa\int h^*(x_1){R}(dx_1|x,\pi^*(x)).
\end{align}
Note that the above equations are in the same form as  ACOEs, and hence, we have that $J(x,\hat{\pi})=\rho(\hat{\pi})=\int h(x_1)\nu(dx_1)$, $\hat{J}(x)=\hat{\rho}=\int \hat{h}(x_1)\nu(dx_1)$, and $J^*(x)=\rho=\int h^*(x_1)\nu(dx_1)$.
We now write
\begin{align*}
J(x,\hat{\pi})-J^*(x)\leq \left|J(x,\hat{\pi})-\hat{J}(x)\right|+\left|\hat{J}(x)-J^*(x)\right|.
\end{align*}
For the first term, we first study the difference $|h(x)-\hat{h}(x)|$:
\begin{align}\label{h_hhat}
&|h(x)-\hat{h}(x)|=|Th(x)-\hat{T}\hat{h}(x)|\nonumber\\
&\leq |Th(x)-T\hat{h}(x)|+|T\hat{h}(x)-Th^*(x)|+|Th^*(x)-\hat{T}h^*(x)|+|\hat{T}h^*(x)-\hat{T}\hat{h}(x)|\nonumber\\
&\leq \kappa \|h-\hat{h}\|_\infty + \kappa\|\hat{h}-h^*\|_\infty +  \|L_c\|_\infty +  \kappa[h]_{H^\delta} \left(\frac{\|L_T\|_\infty}{\kappa}\right)^\delta
+\kappa \|h^*-\hat{h}\|_\infty
\end{align}
where we  used the fact that the used operators are contractions. Furthermore, we used the following upper bound 
\begin{align*}
&|Th^*(x)-\hat{T}h^*(x)| \leq \bigg|c(x,\hat{\pi}(x))+\kappa\int h^*(x_1){R}(dx_1|x,\hat{\pi}(x))\\
&\qquad\qquad\qquad\qquad\quad- \hat{c}(x,\hat{\pi}(x))-\kappa\int h^*(x_1)\hat{{R}}(dx_1|x,\hat{\pi}(x))\bigg|\\
&\leq \|L_c\|_\infty + \kappa [h^*]_{H^\delta} \left(W_1(R(\cdot|x,\hat{\pi}(x)) , \hat{R}(\cdot|x,\hat{\pi}(x)))\right)^\delta\\
&\leq  \|L_c\|_\infty + \kappa [h^*]_{H^\delta} \left(\frac{W_1(\mathcal{T}(\cdot|x,\hat{\pi}(x)) , \hat{\mathcal{T}}(\cdot|x,\hat{\pi}(x)))}{\kappa}\right)^\delta\\
&\leq  \|L_c\|_\infty + \kappa [h^*]_{H^\delta} \left( \frac{\|L_T\|_\infty}{\kappa} \right)^\delta
\end{align*}
where we used Lemma \ref{hold_ineq_lem}. 
We note that by Lemma \ref{holder_lemma}, $[h^*]_{H^\delta} \leq  \frac{K_c D^{1-\delta}}{1-\kappa\left(\frac{K_\mathcal{T}}{\kappa}\right)^\delta}$.
Using Lemma \ref{h_diff_lem}, we also have that 
\begin{align*}
\|h^*-\hat{h}\|_\infty \leq \frac{1}{1-\kappa}  \left(  \|L_c\|_\infty +\left(\frac{K_c D^{1-\delta}  \kappa^{1-\delta}}{1-\kappa\left(\frac{K_\mathcal{T}}{\kappa}\right)^\delta}\right) \left(\|L_T\|_\infty\right)^\delta   \right).
\end{align*}
Combining the above arguments with (\ref{h_hhat}), we obtain
\begin{align*}
\|h-\hat{h}\|_\infty &\leq \frac{1+\kappa}{(1-\kappa)^2}\left(  \|L_c\|_\infty +\left(\frac{K_c D^{1-\delta}  \kappa^{1-\delta}}{1-\kappa\left(\frac{K_\mathcal{T}}{\kappa}\right)^\delta}\right) \left(\|L_T\|_\infty\right)^\delta \right).
\end{align*}
For the difference $\left|\rho(\hat{\pi})-\hat{\rho}\right|$, we then have
\begin{align*}
\left|\rho(\hat{\pi})-\hat{\rho}\right|&= \left|\int h(x_1)\nu(dx_1)-\int \hat{h}(x_1)\nu(dx_1)\right|\leq  \|h-\hat{h}\|_\infty \nu(\mathds{X})\\
&\leq \frac{1+\kappa}{(1-\kappa)}\left(    \|L_c\|_\infty +\left(\frac{K_c D^{1-\delta}  \kappa^{1-\delta}}{1-\kappa\left(\frac{K_\mathcal{T}}{\kappa}\right)^\delta}\right) \left(\|L_T\|_\infty\right)^\delta  \right)
\end{align*}
where we used the identity $\nu(\mathds{X})=1-\kappa$. Finally, combining this bound with the estimates for $|\rho-\hat{\rho}|$  obtained in Theorem \ref{val_dif2}, the proof is complete.
\end{proof}

\noindent{\bf Further Bounds in terms of State-Space Quantization Error.} The performance error of an approximate policy is mainly controlled by the quantities
\begin{align*}
&L_c(x,u):=\left| c(x,u) - \hat{c}(x,u)\right|\nonumber\\
&L_T(x,u):=W_1\left(\mathcal{T}(\cdot|x,u),\hat{\mathcal{T}}(\cdot|x,u)\right).
\end{align*}
This suggests that, in order to minimize the loss, the quantization bins should be chosen so that the deviation of the cost function and the transition kernels is small within each bin.

For the cost function, this can be achieved by constructing quantization bins aligned with the level (or sublevel) sets of 
$c(x,u)$, so that the cost is nearly constant within each bin. However, the transition kernel maps states and actions to the set of probability measures whose pre-image map analysis is far more tedious unless there is further structure on the dynamics. Hence, following the same approach is not effective for the transition kernel. Accordingly, instead of partitioning in terms of the pre-image of the partitions on the space of probability measures, we partition the domain (that is, the state and action sets): The following results show that the error can be made arbitrarily small by choosing a sufficiently high rate of quantization. We define an average loss function $L:\mathds{X}\to\mathds{R}$ induced by the quantization: for $x\in\mathds{X}$ belonging to a quantization bin $B_i$ whose representative state is $y_i$ (i.e., $q(x)=y_i$), a weighted loss function $L(x)$ is defined as
\begin{align}\label{loss_func}
L(x):=\int_{B_i}d_{\mathbb{X}}(x,x') \, \mu_{y_i}(dx').
\end{align}
That is, $L(x)$ can be seen as the mean distance from $x$ to the bin $B_i$ under the measure $\mu_{y_i}$. We denote the uniform bound on the quantization error by $L_{\mathds{X}}$ defined as 
\begin{align*}
L_{\mathds{X}}:=\sup_x L(x).
\end{align*}
We have the following immediate result:
\begin{lemma}\label{lip_lem}
Under Assumption \ref{lip_assmp}, we have that for all $x,u\in\mathds{X}\times\mathds{U}$
\begin{align*}
&L_c(x,u) = |\hat{c}(x,u)-c(x,u)|\leq K_c L_{\mathds{X}},\\
&L_T(x,u)=W_1(\hat{\mathcal{T}}(\cdot|x,u),\mathcal{T}(\cdot|x,u))\leq K_\mathcal{T} L_{\mathds{X}}.
\end{align*}
\end{lemma}


\begin{proof}
Let $x\in B_i$. For the cost difference, we write
\begin{align*}
&|\hat{c}(x,u)-c(x,u)|=\left|\int_{B_i}c(x',u)\mu_{y_i}(dx')-c(x,u)\right|=\left|\int_{B_i}c(x',u)-c(x,u)\mu_{y_i}(dx')\right|\\
& \leq\int_{B_i}K_c d_{\mathbb{X}}(x,x') \mu_{y_i}(dx')\leq K_c L_{\mathds{X}}.
\end{align*}

For the transition difference, for any $\|f\|_{Lip}\leq 1$, we similarly write
\begin{align*}
&\left|\int f(x_1)\hat{\mathcal{T}}(dx_1|x,u)-\int f(x_1)\mathcal{T}(dx_1|x,u)\right|\\
&=\left|\int \int_{B_i} f(x_1)\mathcal{T}(dx_1|x',u)\mu_{y_i}(dx')-\int f(x_1)\mathcal{T}(dx_1|x,u)\right|\\
&\leq \int_{B_i}K_\mathcal{T} {d_{\mathbb{X}}(x,x')} \mu_{y_i}(dx')\leq K_\mathcal{T} L_{\mathds{X}}.
\end{align*}
\end{proof}

\begin{corollary}\label{pol_err_cor}{(of Theorem \ref{pol_err})}
Suppose $\hat{\pi}$ satisfies (\ref{acoe_finite}). Under Assumptions \ref{mix_kernel} and   \ref{lip_assmp}, for any $\delta$ such that $K_\mathcal{T}<\left(\frac{1}{\kappa}\right)^{\frac{1-\delta}{\delta}}$, we have that
\begin{align*}
\rho(\hat{\pi})-\rho\leq    \frac{2K_c D^{1-\delta}}{(1-\kappa)\left(1-\kappa \left(\frac{K_\mathcal{T}}{\kappa}\right)^\delta\right)}L_\mathds{X}^\delta.
\end{align*}
\end{corollary}

\begin{remark} 
The result is stated for finite action spaces. However, we can easily get a further upper bound for continuous action spaces by combining Theorem \ref{weak:mainthm2} and Theorem \ref{pol_err} (or Corollary \ref{pol_err_cor}). Let $\hat{\pi}$ denote the policy designed for discretized action and state spaces, and let $\rho(\hat{\pi})$ denote the average cost we would receive if we applied the policy $\hat{\pi}$ to the original model. $\rho$ is the optimal value for the original state and action spaces, and ${\rho'}$ denotes the value of the problem with a discretized action space and a continuous state space. We can then write
\begin{align*}
\rho(\hat{\pi}) - \rho \leq \left|\rho(\hat{\pi}) - {\rho'}\right| + \left| {\rho'} - \rho\right| \leq  \frac{2K_c D^{1-\delta}}{(1-\kappa)\left(1-\kappa \left(\frac{K_\mathcal{T}}{\kappa}\right)^\delta\right)}L_\mathds{X}^\delta +  \frac{K_cD^{1-\delta}}{1-\kappa\left(\frac{K_\mathcal{T}}{\kappa}\right)^\delta}\E(\Lambda)^\delta
\end{align*}
where we have used Corollary \ref{pol_err_cor} for the first term, and Theorem \ref{weak:mainthm2}(iii)  for the second term (see (\ref{act_error}) for the definition of $\E(\Lambda)$).
\end{remark}

\subsubsection{Finite Approximations via Weak Continuity and Asymptotic Optimality}
\adk{In Section \ref{FinMDPWassM}, we used Assumption \ref{lip_assmp} to obtain quantitative bounds on approximation errors using the Lipschitz continuity of the cost function and the transition kernel. In this section, we relax the Lipschitz continuity assumption and focus on asymptotic analysis only under continuity of the cost, and weak continuity of the kernel. We emphasize that Lipschitz continuity conditions are generally more restrictive; for example in the belief-MDP reduction for partially observable Markov Decision Processes (POMDPs), there is a significant difference between models which allow for weak Feller continuity vs. those which satisfy Lipschitz regularity (see the tutorial paper \cite{tutorialkara2024partially}); an analogous situation applies for problems where lifting to a larger space (such as probability measure valued dynamics) is utilized as in decentralized stochastic control problems, as well as mean-field models}. In this subsection, we work with Assumption \ref{weak:as} specialized to finite action spaces:
\begin{assumption}
\label{weak:as1}
\begin{itemize}
\item [(a)] $\mathbb{U}$ is finite.
\item [(b)] $\mathbb{X}$ is compact.
\item [(c)] The one-stage cost function $c$ is bounded and continuous.
\item [(d)] The stochastic kernel ${\cal T}(\,\cdot\,|x,u)$ is weakly continuous in $(x,u) \in {\mathbb X} \times \mathds{U}$ (that is, weak Feller continuous).
\end{itemize}
\end{assumption}
In this section, we will apply a uniform quantization of the compact state space with bin diameter $\frac{1}{n}$ so that 
\begin{align*}
L_\mathds{X}=\sup_x L(x) \leq \frac{1}{n}
\end{align*}
where $L(x)$ is defined in (\ref{loss_func}).

In Section 4.2.2 of \cite{SaLiYuSpringer} total variation continuity was imposed for near optimality of quantized models under the average cost criterion. In the previous subsection, this was relaxed to Wasserstein continuity. The convergence result along the same lines can be shown to work under only Assumptions \ref{mix_kernel} and \ref{weak:as1}, albeit without a modulus of continuity. 
\begin{theorem}\label{weakContApprAsy}
Let $\pi_n$ denote an optimal policy for the quantized model which solves the ACOE for the approximate model with discretization width $1/n$.
For a given initial state $x_0\in\mathds{X}$, let ${J}(x_0,\pi_n)$ denote the average cost of the policy $\pi_n$ evaluated under the original model and let $J^*(x_0)$ denote the optimal average cost.

Under Assumption \ref{mix_kernel}, these values are independent of the initial state; we write $\rho(\pi_n)={J}(x_0,\pi_n)$ and $\rho=J^*(x_0)$ for all $x_0\in\mathds{X}$.
Then, under Assumptions  \ref{mix_kernel} and \ref{weak:as1}, we have
\[\lim_{n\rightarrow\infty} \rho(\pi_n) =\rho.\]
\end{theorem}

\begin{proof}
The  proof relies on the contraction operators and the corresponding fixed point equations used in the proof of Theorem  \ref{pol_err}, i.e. (\ref{operators2}). In particular, we consider the fixed point equations:
\begin{align*}
&h_n(x)=c(x,{\pi_n}(x))+\kappa\int h_n(x_1){R}(dx_1|x,{\pi_n}(x))\\
&h^*(x)=c(x,\pi^*(x))+\kappa\int h^*(x_1){R}(dx_1|x,\pi^*(x))
\end{align*}
where $R(\cdot|x,u):=\frac{\mathcal{T}(\cdot|x,u)-\nu(\cdot)}{\kappa}$ with $\kappa=1-\nu(\cdot)<1$. Note that $\pi_n$ is an optimal policy for the approximate cost function $c_n$ and the approximate kernel $\mathcal{T}_n$ under the average cost criterion and equivalently under a discounted cost criterion with discount factor $\kappa$. Furthermore, the  fixed point $h_n$ corresponds to the discounted cost value of the policy $\pi_n$ under the original model with a discount factor $\kappa$, and $h^*$ corresponds to the optimal discounted value function for the original dynamics under the discount factor $\kappa$.

We can then use Theorem 4.4 of \cite{kara2020robustness} to show that $h_n\to h^*$ point-wise under Assumption \ref{weak:as1}, if 
\begin{itemize}
\item ${R}_n(\cdot|x_n,u)\to R(\cdot|x,u)$ weakly for any $x_n\to x$ and for all $u$
\item $c_n(x_n,u)\to c(x,u)$ for any $x_n\to x$ and for all $u$.
\end{itemize} 
For the first item, it is equivalent to show that $\mathcal{T}_n(\cdot|x_n,u)\to \mathcal{T}(\cdot|x,u)$ weakly. For the cost function we have that:
\begin{align*}
c_n(x_n,u) = \int_{B_{n,i}} c(z,u)\mu_n (dz)
\end{align*}
where $B_{n,i}$ denotes the quantization bin $x_n$ belongs to and $\mu_n$ is the weight measure $\mu$ concentrated on the set $B_{i,n}$. Thus, we need to show that for any fixed $\epsilon>0$,  we can find a large enough $N<\infty$ such that for $n>N$, we have that $\left|c(z,u) - c(x,u)\right|< \epsilon$   for all $z\in B_{i,n}$.

For fixed $\epsilon >0$, we can find $\delta>0$ such that $|c(x,u) -c(z,u)|< \epsilon $ for all $d_\mathds{X}(x,z)<\delta$ since $c(x,u)$ is continuous by assumption. Thus, we now want to find a sufficiently large $N<\infty$ such that for such a $\delta$, $d_{\mathds{X}}(x,z)< \delta$ for all $z\in B_{i,n}$ for $n\geq N$. Recall that $B_{n,i}$ represents the quantization bin $x_n$ belongs to, and by construction we have that $d_\mathds{X}(z,x_n)\leq \frac{1}{n}$ for all $z\in B_{i,n}$ which can be made smaller than $\delta/2$ for all $n\geq N_1$ for a sufficiently large $N_1$. Furthermore, $x_n\to x$ by assumption, and thus we can make $d_\mathds{X}(x_n,x)< \delta/2$ for all $n\geq N_2$ for some other sufficiently large $N_2$. Picking the greater of $N_1$ and $N_2$ implies that 
\begin{align*}
d_{\mathds{X}}(x,z) \leq d_\mathds{X}(x,x_n) + d_\mathds{X}(x_n,z) < \delta
\end{align*} 
for all $n\geq \max(N_1,N_2)$ and for all $z\in B_{i,n}$, which proves the claim that $c_n(x_n,u)\to c(x,u)$ for all $x_n \to x$.

Using identical arguments and noting that for any continuous and bounded $f$, we have that 
\begin{align*}
&\left| \int f(x_1)\mathcal{T}_n(dx_1|x_n,u) - \int f(x_1)\mathcal{T}(dx_1|x,u)\right|\\
&=\left| \int_{B_{n,i}} \int f(x_1)\mathcal{T}_n(dx_1|z,u)\mu_n(dz)  - \int f(x_1)\mathcal{T}(dx_1|x,u)\right|
\end{align*}
we can also conclude that $\mathcal{T}_n(\cdot|x_n,u)\to \mathcal{T}(\cdot|x,u)$ weakly for all $x_n \to x$. By Theorem 4.4 of \cite{kara2020robustness}, this shows that $h_n\to h^*$. 

We conclude the proof by using the Dominated Convergence Theorem and noting that $\rho(\pi_n)=\int h_n(x_1)\nu(dx_1)$ and $\rho=\int h^*(x_1)\nu(dx_1)$.
\end{proof}

We now summarize the main results of this section as follows:
\begin{theorem}\label{main_thm}
Suppose $\hat{\pi}$ satisfies (\ref{acoe_finite}) and is optimal for the approximate model. For a given initial state $x_0\in\mathds{X}$, let $J(x_0,\hat{\pi})$ denote the average cost of $\hat{\pi}$ in the original MDP, $J^*(x_0)$ the optimal average cost of the original MDP, and $\hat{J}(x_0)$ the optimal average cost of the approximate MDP.

Under Assumption \ref{mix_kernel}, these values are independent of the initial state; we write $\rho(\hat{\pi}):=J(x_0,\hat{\pi})$, $\hat{\rho}:=\hat{J}(x_0)$, and $\rho:=J^*(x_0)$ for all $x_0\in\mathds{X}$.

Fix $\delta$ such that $K_\mathcal{T}<\left(\frac{1}{\kappa}\right)^{\frac{1-\delta}{\delta}}$ where $\kappa:=1-\nu(\mathds{X})>0$, and let $D$ denote the diameter of the state space.
\begin{itemize}
\item[i.] Under Assumption \ref{lip_assmp} with $K_\mathcal{T}<1$,
\begin{align*}
|\rho-\hat{\rho}|\leq  \frac{K_c}{1-K_\mathcal{T}}L_\mathds{X}.
\end{align*}
\item[ii.] Under Assumptions \ref{mix_kernel} and \ref{lip_assmp},
\begin{align*}
\left|{\rho}-\hat{\rho}\right|\leq\frac{K_cD^{1-\delta}}{1-K_\mathcal{T}^\delta\kappa^{1-\delta}}\left(L_\mathds{X}\right)^\delta
\end{align*}
and furthermore,
\begin{align*}
\rho(\hat{\pi})-\rho\leq \frac{2K_cD^{1-\delta}}{(1-\kappa)\left(1-K_\mathcal{T}^\delta\kappa^{1-\delta}\right)}\left(L_\mathds{X}\right)^\delta.
\end{align*}
\item[iii.] If $L_\mathds{X}=\frac{1}{n}$ (which is possible since $\mathds{X}$ is compact), then denoting the learned policy by $\pi_n$, under Assumptions \ref{mix_kernel} and \ref{weak:as1},
\[\lim_{n\rightarrow\infty} \rho(\pi_n) =\rho.\]
\end{itemize}
\end{theorem}

Now that we have presented the contraction framework needed for our analysis, we will move on to establishing a Q-learning algorithm and its convergence to near-optimality.

\section{Quantized Q-Learning for Continuous Spaces under the Infinite Horizon Average Cost Criterion}\label{q_cont_sec}

In this section, we study Q-learning for a continuous MDP under a discretization map, and investigate whether the algorithm converges and, if so, whether its limit corresponds to the approximate MDP models constructed in the previous sections.

In particular we will present synchronous and asynchronous Q-learning algorithms for systems with continuous spaces and show the convergence of these algorithms to the value functions of appropriately defined finite MDP models consistent with those constructed in Section \ref{finite_state_sec}.  

We denote by $Q^*:\mathds{Y}\times\mathds{U}\to \mathds{R}$ the Q-values (factors, or functions) for the finite model introduced in Section \ref{finite_state_sec} for some weight measure $\mu\in\P(\mathds{X})$. For any discretized state $y_i\in\mathds{Y}$ and any control action $u\in\mathds{U}$, the Q value of the pair $(y_i,u)$ satisfies the following equality:
\begin{align}\label{finite_Q}
\hat{\rho} + Q^*(y_i,u)=C(y_i,u)+ \sum_{y_j\in\mathds{Y}} P(y_j|y_i,u)\min_{v \in \mathds{U}}Q^*(y_j,v)
\end{align}
where $P$ and $C$ are defined in (\ref{finite_cost}) and $\hat{\rho}$ is the value of the average cost problem for the finite model.

Theorem \ref{pol_err} provides bounds on the performance loss of the policies designed for discretized models (and Theorem \ref{weakContApprAsy} generalizes this to the case with only weakly continuous models, though with only asymptotic convergence). Hence, if we can find iterations that converge to the Q values in (\ref{finite_Q}), we will be able to obtain performance results through Theorems \ref{pol_err} and \ref{weakContApprAsy}.

In the following, we present a synchronous and an asynchronous algorithm. It is important to note that the quantized process $q(X_t)$ is not an MDP, and in fact should be viewed as a POMDP, a perspective used by \cite{kara2021convergence} \citep[see also][]{KSYContQLearning}. 

For the asynchronous setup, we work with a single trajectory of the original process, and the data are given sequentially so that we cannot rely on the Markov properties. Accordingly, we generalize the proof method given by \cite{kara2021convergence} for the average cost criterion and impose ergodicity properties under an exploration policy.

\subsection{Synchronous Quantized Q-Learning for Continuous Space Average Cost MDPs}\label{syn_q_sec} 

We first present a synchronous Q-learning algorithm. As noted above, the proof is more direct in this case. In the following, we present the synchronous Q-learning algorithm, which we will use in this section. Note that the discretization part of the algorithm follows the same steps as introduced in Section \ref{finite_state_sec}.

Recall the quantization of the state space with a collection of disjoint sets $\{B_i\}_{i=0}^{M-1}$ such that $\bigcup_i B_i=\mathds{X}$, and $B_i\bigcap B_j =\emptyset$ for any $i\neq j$ with a representative state, $y_i\in B_i$, for each disjoint set. Denote the new finite state space by
$\mathds{Y}:=\{y_0,y_1,\dots,y_{M-1}\}$.

\begin{algorithm}[H]
\caption{Synchronous Quantized Q-Learning}
\label{Qit1}
\begin{algorithmic}
\STATE Input: $Q_0$, $\mathds{Y}$, $\mathds{U}$, $\{\mu_y\}_{y\in\mathds{Y}}$, reference $(y_0,u_0)$
\FOR{$t=0,\ldots,L-1$}
\FOR{each $(y_i,u_j)\in\mathds{Y}\times\mathds{U}$}
\STATE Sample $x_i\sim\mu_{y_i}$, $X_1^{i,j}\sim\mathcal{T}(\cdot|x_i,u_j)$
\STATE Update
\begin{align}\label{q_it_syn}
Q_{t+1}(y_i,u_j)=(1-\alpha_t)Q_t(y_i,u_j)
+\alpha_t\Big(c(x_i,u_j)+\min_v Q_t(q(X_1^{i,j}),v)\Big)
\end{align}
\ENDFOR
\STATE Normalize $\hat{Q}_{t+1}(y_i,u_j)= Q_{t+1}(y_i,u_j)-Q_{t+1}(y_0,u_0)$

\ENDFOR
\RETURN $\hat Q_L$
\end{algorithmic}
\end{algorithm}

The algorithm updates the Q-values for all $(y,u)\in\mathds{Y}\times\mathds{U}$ pairs synchronously. To perform the updates for a discrete state $y_i$ (or quantization bin $B_i$) and an action $u_j$,  a \emph{continuous} state is sampled from the bin $B_i$ according to a pre-determined measure $x_i\sim\mu_{y_i}(\cdot)$. Along with this sampled state, the corresponding future state  $X_1^{ij}$ is also taken from the dataset. The update in (\ref{q_it_syn}) then uses the future continuous state $X_1^{ij}$ and the cost $c(x_i,u_j)$ corresponding to the sampled state $x_i$ under the action $u_j$. Finally, the Q-values are normalized by subtracting $Q_t(y_0,u_0)$ for a pre-selected reference pair $(y_0,u_0)$ from all Q-values.




The next result shows that these iterations converge to the Q values given in (\ref{finite_Q}) for an appropriate weight measure. 

\begin{theorem}\label{syn_th}

We assume that the approximate transition probabilities for the finite model satisfy for all $y,y',u,u'$
\begin{align}\label{sp_cont}
\frac{1}{2}\sum_j \left|P(y_j|y,u) - P(y_j|y',u')\right| \leq \beta <1 .
\end{align}
If $\alpha_t=\frac{1}{t+1}$ and the iterations are given by (\ref{q_it_syn}), then $Q_t$ converges to $Q^*$ (see (\ref{finite_Q})) under the span semi-norm, and $\hat{Q}_t$ converges to $\hat{Q}^*$ under the uniform norm where $\hat{Q}(y,u)=Q(y,u)-Q(y_0,u_0)$ for a predetermined pair $(y_0,u_0)$. We also have that
\begin{align*}
\hat{\rho}+\hat{V}^*(y_0)=C(y_0,u_0)+\sum_{y_1\in\mathds{Y}} \hat{V}^*(y_1) P(y_1|y_0,u_0)
\end{align*}
where $\hat{V}^*(y)=\min_{u}\hat{Q}^*(y,u)$ and where $\hat{\rho}$ is the value of the approximate model introduced in Section \ref{finite_state_sec} under the average cost  criterion. Accordingly, the results of Theorem \ref{main_thm} follow. 
\end{theorem}

A proof is given in Appendix \ref{proofSynQQL}. We note that under Assumption \ref{mix_kernel}, (\ref{sp_cont}) always holds. In particular, if $\mathcal{T}(\cdot|x,u)\geq \nu(\cdot)$, then Lemma \ref{minor_fin_lemma} implies that $P(y_j|y,u)>0$ for some $y_j$ and for all $y,u$. This then implies the condition (\ref{sp_cont}). On the other hand, there may be settings where the original problem does not satisfy the minorization condition in Assumption \ref{mix_kernel}, however, one may find discretization schemes which do satisfy condition (\ref{sp_cont}). Indeed, Example \ref{minor_counter} is a case where the original model does not satisfy Assumption \ref{mix_kernel}; however, one can always find a discretization which does satisfy (\ref{sp_cont}). Accordingly, for the convergence result, we impose the weaker condition, that is (\ref{sp_cont}). Nonetheless, to show the near optimality of the learned policies, we still need Assumption \ref{mix_kernel}.

\subsection{Asynchronous Quantized Q-Learning for Continuous Space Average Cost MDPs}\label{asyn_q_sec}

The Q-learning algorithm we presented earlier is constructed under the assumption that we can generate samples from every quantization bin synchronously. We now present an algorithm that  learns the Q-values and an optimal policy of the finite model constructed in Section \ref{finite_state_sec}, from a single trajectory. In this setting, the main challenge is that only a single trajectory of the original continuous state MDP is available. The question is whether  Q-learning applied with the quantized process converges.

The decision maker applies an arbitrary admissible policy $\pi$  and collects realizations of quantized states, actions, and stage-wise cost under this policy:
$
Y_0,U_0,c(X_0,U_0),Y_1,U_1,c(X_1,U_1) \ldots 
$
where $Y_t$ denotes the representative (quantized) state corresponding to $X_t$.

\begin{assumption}\label{mix_kernel2}
For the finite model transition probabilities $P(y_j|y,u)$ (see (\ref{finite_cost})), where the weighting measure is given by the invariant measure of the state process under the exploration policy, there exists a bin $B_j$, with representative state $y_j$, such that
\[P(y_j|y,u) >0,\]
 for every $y \in \mathds{Y}$ and $u\in\mathds{U}$. Furthermore, the index $j$ of this bin is known. 
\end{assumption}


We propose a shifted Q-learning algorithm by subtracting the value $\delta V_t(y_j)$ for some sufficiently small  $\delta>0$, where $V_t(y')=\min_uQ_t(y',u)$. This yields, for all $(y,u)\in\mathds{Y}\times\mathds{U}$
\begin{align}\label{asyn_q_it}
&Q_{t+1}(y,u)=(1-\alpha_t(y,u)) \, Q_t(y,u)+\alpha_t(y,u)\left(c(X_t,U_t)+ \min_{v \in \mathds{U}} Q_t(Y_{t+1},v)-\delta V_t(y_j)\right).
\end{align}



\begin{algorithm}[H]
\caption{Asynchronous Quantized Q-Learning}
\label{Qit_async}
\begin{algorithmic}
\STATE \textbf{Input:} $Q_0$, quantizer $q:\mathds{X}\to\mathds{Y}$, exploration policy $\pi$, horizon $L$, parameter $\delta>0$
\STATE Initialize $N(y,u)=0$ for all $(y,u)\in\mathds{Y}\times\mathds{U}$
\STATE Set $Q_0$

\FOR{$t=0,\ldots,L-1$}
\STATE Observe $(X_t,U_t,c(X_t,U_t),X_{t+1})$
\STATE $y \gets q(X_t)$, \quad $y' \gets q(X_{t+1})$
\STATE $N(y,U_t) \gets N(y,U_t)+1$
\STATE $\alpha_t(y,U_t) \gets \frac{1}{1+N(y,U_t)}$
\STATE $Q_{t+1}(y,U_t) \gets (1-\alpha_t(y,U_t))Q_t(y,U_t)+\,\alpha_t(y,U_t)\!\left(c(X_t,U_t)+\min_{v}Q_t(y',v)-\delta V_t(y_j)\right)$
\STATE $U_{t+1} \sim \pi(\cdot|y')$
\ENDFOR
\RETURN{$Q_L$}
\end{algorithmic}
\end{algorithm}
We impose the following assumptions for convergence

\smallskip

\begin{assumption}\label{asyn_q}
\hfill
\begin{itemize}
\item [(1.)] We set $\alpha_t(y,u)=0$ unless $(Y_t,U_t)=(y,u)$. Otherwise, let
\[\alpha_t(y,u) = {1 \over 1+ \sum_{k=0}^{t} 1_{\{Y_k=y, U_k=u\}} }.\]
\item [2.] Under the (memoryless or stationary) exploration policy $\pi(\cdot|\cdot)$,  $X_t$ is uniquely ergodic and thus has a unique invariant measure $\mu_{\pi}$ such that $\mu_\pi(B_i)>0$ for every quantization bin $B_i$. 
\item [3.] The exploration policy $\pi$ has full support over actions at every quantized state: $\pi(u|y)>0$ for all $(y,u)\in\mathds{Y}\times\mathds{U}$.
\end{itemize}
\end{assumption}

We note that a sufficient condition for the second item in Assumption~\ref{asyn_q} is that the state process $\{X_t\}_{t\geq0}$ is positive Harris recurrent under the exploration policy.  In particular, together with the minorization condition in Assumption \ref{mix_kernel2}, the process becomes positive Harris recurrent. Item~(3) is a  full support exploration condition. It is satisfied, for example, by any $\epsilon$-greedy policy with $\epsilon>0$ or by a uniformly randomized policy over the finite action set $\mathds{U}$. Together, Items~(2) and~(3) ensure that every quantized state-action pair $(y,u)\in\mathds{Y}\times\mathds{U}$ is visited infinitely often almost surely, which is required both for the step-size conditions of Lemma~\ref{czaba_lemma} and for the empirical-average convergence in the proof of Theorem~\ref{quant_q}.

For the discounted cost criterion, an analogous result was proven by \cite{kara2021convergence} (see also related results by  \cite{CsabaSmart} and \cite{singh1994learning}).
We first recall a key result by \citet[Prop. 4.5]{BertsekasTsitsiklis} and by \citet[Lemma 1]{singh2000convergence}. 
\begin{lemma}\label{czaba_lemma}\cite[Prop. 4.5]{BertsekasTsitsiklis}\cite[Lemma 1]{singh2000convergence}
Consider a stochastic process $(\alpha_t,\Delta_t,F_t),$ $t\geq 0$, where $\alpha_t,\Delta_t,F_t:\mathds{S}\to\mathds{R}$ for some finite set $\mathds{S}$ and satisfy the equations
\begin{align*}
\Delta_{t+1}(s)=(1-\alpha_t(s))\Delta_t(x)+\alpha_t(s)F_t(s)
\end{align*}
Let $P_t$ be a sequence of increasing $\sigma$-fields such that $\alpha_0$ and $\Delta_0$ are $P_0$ measurable and $\alpha_t,\Delta_t,F_{t-1}$ are $P_t$ measurable.  Assume that the following hold:
\begin{itemize}
\item $\sum_t \alpha_t(s)=\infty$, $\sum_t \alpha_t^2(s)<\infty$ almost surely,
\item $\left|E[F_t(\cdot)|P_t]\right|_\infty \leq \beta\|\Delta_t\|_\infty + c_t$ where $\beta<1$ and $c_t$ converges to zero almost surely.
\item $Var(F_t(s)|P_t)\leq K(1+\|\Delta_t\|_\infty)^2$ for some constant $K<\infty$.
\end{itemize}
Then $\Delta_t$ converges to zero almost surely.
\end{lemma}

Convergence of asynchronous methods is challenging because the quantized state variable does not satisfy the Markov property and the error term is not a martingale noise. This issue does not arise in the synchronous method due to independent sampling across iterations. To deal with this, we utilize the ergodicity assumption on the state process and decompose the non-Markov noise term so that it can be written as the difference between the stationary mean and the empirical average. This difference is then folded into the decaying error term, i.e., the $c_t$  term in Lemma~\ref{czaba_lemma}.

 \begin{theorem}\label{quant_q}
 Under Assumption \ref{mix_kernel2} and Assumption \ref{asyn_q}, the algorithm given in (\ref{asyn_q_it}) converges almost surely, for sufficiently small $\delta$ (specifically, $\delta<\min_{y,u\in\mathds{Y}\times\mathds{U}} P(y_j|y,u)$ where $y_j$ is as in Assumption \ref{mix_kernel2}), to $Q^*$ which satisfies
\begin{align}\label{fixed2}
\hat{\rho} + Q^*(y,u)=C(y,u)+\sum_{z \in \mathds{Y}} P(z|y,u)\min_{v \in \mathds{U}} Q^*(z,v).
\end{align}
$P$ and $C$ are defined in (\ref{finite_cost}) with the weighting measure being the stationary distribution of the state process $X_t$ under the exploration policy.

Furthermore, for the learned value function and the policy, near optimality and convergence results of Theorem \ref{main_thm} are applicable.
\end{theorem}


\begin{proof}
The proof can be found in Appendix \ref{quant_q_proof}.
\end{proof}


\sy{

\section{Extensions, Reflections, and Refinements}

In this section, we present some reflections and refinements. To avoid creating a distraction from the main text, these have been summarized in this final section. The first one involves arriving at near optimality via discounted cost minimization with discount parameters close to $1$, the second involves the relaxation of the minorization condition at the expense of further regularity in the kernel, and the third discussion is on empirical quantized model learning and equivalence with the quantized Q-learning convergence results.

\subsection{Approximate Optimality via Discounted Cost Criterion Approximation: Beyond Minorization}

Consider the following two conditions: (i) There exists a solution to the average cost optimality equation, and (ii) this solution is obtained via the vanishing discount method. Under these conditions, it follows  \citep[see][Theorems 1 and 2]{creggZeroDelayNoiseless}  \cite[see also][Theorem 7.3.6]{yuksel2020control} that a near-optimal policy for the discounted cost problem is also near-optimal for the average cost problem. 

Accordingly, a further method would be to approximate the Q-learning algorithm by its classical discounted version. This approach, in particular, does not make explicit use of the ergodicity or minorization condition (\ref{mix_kernel}) (except for its use in establishing the existence of a solution to the average cost optimality equation in our paper; this condition is not necessary in general). For example, for the belief-MDP reduction of partially observable models, such a condition does not hold; yet one can establish conditions under which a solution to the average cost optimality equation exists and is obtained via the vanishing discount technique (see, e.g., \cite{Bor00,StettnerSICON19,runggaldier1994approximations,platzman1980optimal,hsu2006existence,demirci2023average}). However, the question of approximation rates as the discount parameter approaches unity remains open for such problems. We leave this important direction for future research, but only note that one can arrive at near average cost optimality via such a general method under these relatively mild conditions as well. 

\subsection{Near Optimality of Finite Approximations when ACOE holds but Minorization Does Not}

For the difference of the value functions between the original and the approximate model, we were able to relax Assumption \ref{mix_kernel}, and use Assumption \ref{val_dif} with $K_\mathcal{T}<1$ instead. However, we could not use the same method to show near optimality of the approximate model policies. We now show that under additional regularity conditions, but not requiring minorization, we can do this for the asymptotic analysis; that is, we can show that the policy designed for the discretized model when applied to the original model is near optimal for the original model, with the performance loss decaying to zero as the quantization gets finer. Such conditions are particularly relevant for controlled (continuous-time) diffusions which are time-discretized. We make the following assumption (as in Theorem 4.1 of \cite{yuksel2023borkar}): 

\begin{assumption}\label{WeakCompInvEqui}
Suppose that 
\begin{itemize}
\item[(i)]
We have $\mathbb{X} \subset \mathbb{R}^d$ for some finite $d$ and $\mathds{X}$ is compact, and $\mathbb{U}$ is finite.
\item[(ii)]
For every stationary policy $\pi \in \Pi_{S}$ there exists a unique invariant probability measure. 

\item[(iii)] The kernel ${\cal T}(dy | x,u)$ is such that, the family of conditional probability measures ${\cal T}(dy | x,u), x \in \mathbb{X}, u\in \mathbb{U}\}$ admit densities $f_{x,u}(y)$ with respect to a reference measure $\psi$, and all such densities are bounded and equicontinuous (over $x \in \mathbb{X}, u\in \mathbb{U}$). 

\item[(iv)] ${\cal T}$ is weak Feller continuous. 
 %
\end{itemize}
\end{assumption}

\begin{theorem}\label{young_conv}
Suppose that Assumption \ref{lip_assmp} with $K_\mathcal{T}<1$ holds, and Assumption \ref{WeakCompInvEqui} holds. Then, 
\[\lim_{n\to\infty}J(x,\pi_n)  = J^*(x).\]
\end{theorem}

\begin{proof}
The proof can be found in Appendix \ref{young_conv_proof}.
\end{proof}

\subsection{Equivalence with Empirical Model Learning}

The implication of Theorem \ref{quant_q} is that the Q-learning algorithm (\ref{asyn_q_it}) converges to a limit where the limit solves the optimality equation (\ref{fixed2}) for an approximate model defined by the stationary distribution of the state process $X_t$ under the exploration policy. In particular, this model is precisely the empirical limit of a Markovian approximation in the following sense: Let the exploration policy $\pi$ given in the (quantized) Q-learning algorithm give rise to the invariant probability measure $\mu_{\pi}^*$. The limiting Q-function $Q^*( y, u )$ corresponds to the optimal Q-function of an approximate MDP defined over the quantized state space $\mathbb{Y}$. The effective cost $C^*( y, u )$ is the average cost over the bin $B_y$ weighted by the invariant distribution $\mu_{\pi}^*$ conditioned on bin $B_y$:
\begin{equation}\label{modelApp}
C^*( y, u ) = \mathbb{E}_{ x \sim \mu_{\pi}^* | x \in B_y } [ c( x, u ) ] = \int_{B_y} \frac{\mu_{\pi}^*(dx)}{\mu_{\pi}^*(B_y)}c(x,u).
\end{equation}
Observe that the above is, \cite[see e.g.][Theorem 2.1]{karayukselNonMarkovian}, equal to the almost sure limit of the empirical expression on the right hand side below:
\begin{equation}\label{empCostEst}
C^*( y, u )  = \lim_{N \to \infty} \frac{\sum_{k=0}^{N-1} c(X_k,U_k) 1_{\{X_k \in B_y,U_k=u\}}}{ \sum_{k=0}^{N-1} 1_{\{X_k \in B_y,U_k=u\}}}. 
\end{equation}
Similarly, the effective transition probability $P^*( y' | y, u )$ represents the probability of transitioning from bin $B_y$ to bin $B_{ y' }$ under action $u$, averaged over the invariant distribution:
\begin{equation}\label{costApp}
P^*( y' | y, u ) = \mathbb{P}_{ x \sim \mu_{\pi}^* | x \in B_y } [ q( X_{ t+1 } ) = y' | X_t = x, U_t = u ] =  \int_{B_y} \frac{\mu_{\pi}^*(dx)}{\mu_{\pi}^*(B_y)} {\cal T}(B_{y'} | x,u).
\end{equation}

Likewise, by \cite[Theorem 2.1]{karayukselNonMarkovian}, the above is the almost sure empirical limit of the right hand side below:
\begin{equation}\label{empTranEst}
P^*( y' | y, u ) = \lim_{N \to \infty} \frac{\sum_{k=0}^{N-1} 1_{\{X_{k+1} \in B_{y'}\} } 1_{\{X_k \in B_y,U_k=u\}}}{ \sum_{k=0}^{N-1} 1_{\{X_k \in B_y,U_k=u\}}} 
\end{equation}

An interpretation of the above result then is that one can first obtain the approximate model given with (\ref{modelApp}-\ref{costApp}) by forcing the data into a Markovian model for both the empirical cost estimate (\ref{empCostEst}) and empirical transition kernel estimate (\ref{empTranEst}), and then solve the MDP as if this empirically constructed model is the actual one, instead of running Q-learning whose limit is then optimal precisely for this learned/empirically constructed model. Accordingly, for this learned value function and the policy, near optimality and convergence results of Propositions~\ref{val_dif2} and~\ref{val_dif}, and Theorems~\ref{pol_err} and~\ref{weakContApprAsy} are applicable.

}

\section{Simulation and Case Study}

We consider a continuous space control problem. We let the state space  be $\mathds{X}=[0,1]$, and the action space to be $\mathds{U}=[-1,1]$.  The stage wise cost function is given by
\begin{align*}
c(x,u)=0.7(1-x) + 0.2(u+1)
\end{align*}
For the dynamics we assume: if $u>0$, for a given state $x$
\begin{align*}
x_1\sim\begin{cases} &\rm{Unif}[x,\min(x+u,1)] \text{ w.p. } 0.9\\
& \rm{Unif}[0,1] \text{ w.p. } 0.1.
\end{cases}
\end{align*}
If $u\leq 0$ 
\begin{align*}
x_1\sim\begin{cases} &\rm{Unif}[\max(x+u,0),x] \text{ w.p. } 0.9\\
& \rm{Unif}[0,1] \text{ w.p. } 0.1
\end{cases}
\end{align*}

The action controls the direction and magnitude of the stochastic state transitions: positive actions try to move the state upward and negative actions tend to move it downward, with larger actions leading to more likely movements. The cost function favors states closer to the upper end of the state space and penalizes large control inputs, and thus creates a tradeoff between keeping a high quality state and limiting control effort.

For  exploration, we use a randomized control policy  such that $u_t\sim \text{Unif}[0,1]$ for all $t$. We will analyze the problem numerically after the discretization of the action space. We verify Assumption \ref{lip_assmp} and Assumption \ref{mix_kernel}. Assumption \ref{mix_kernel} is satisfied with $\nu(\cdot)=0.1 \times \text{Unif}[0,1]$. For Assumption \ref{lip_assmp}, the Lipschitz constant of $c(x,u)$ is $0.7$ since the cost is linear in the state variable. 
For the transition kernel, the first order Wasserstein distance is equal to the $L_1$ distance of the CDF functions. Hence, one can check that the Lipschitz constant of the kernel is bounded by $0.9$.

Figure~\ref{rel_val_func} shows the convergence behavior of the relative value function for both the synchronous and asynchronous algorithms when the state and action spaces are discretized with 5 discretization bins.  Recall that the relative value function that satisfies the ACOE is not unique and any shifted version of the function satisfies the ACOE. Hence, for better comparison we normalize them by subtracting the minimum value the functions when we plot them. The limit values of the asynchronous and synchronous algorithms differ slightly because the weight measure $\mu$ (see (\ref{norm_inv})) is determined differently in each case. For the asynchronous algorithm, $\mu$ is the invariant measure of the state process under the exploration policy, whereas for the synchronous algorithm, $\mu$ is a user-specified sampling distribution over each bin; for the plotted values, we used the uniform measure.
\begin{figure}[h!]
\begin{center}
\includegraphics[scale=0.43]{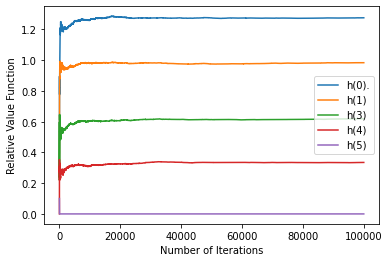} \includegraphics[scale=0.43]{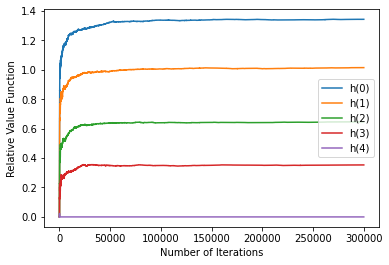} 
\caption{Relative value function convergence for synchronous(left) and asynchronous(right) algorithms}\label{rel_val_func}
\end{center}
\end{figure}

 Figure \ref{fig1} shows the convergence of the value constant, i.e. $\delta V_t(y_j)$, for different initial values $x_0=0.3,0.5,0.8$ for the asynchronous algorithm. We use the discretized action values $\hat{\mathds{U}}=\{-1,0,1\}$, and for the state space we choose the bins to be the intervals $[0,0.25], (0.25,0.5], (0.5,0.75], (0.75,1]$. As it can be seen from the figure, and as expected, the limit value constant does not depend on the initial state.
\begin{figure}[h!]
\begin{center}
\includegraphics[scale=0.5]{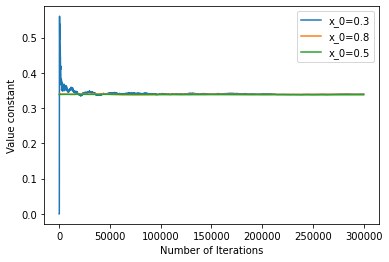}
\caption{Algorithm convergence under different initial conditions}\label{fig1}
\end{center}
\end{figure}

In Figure \ref{fig2},the upper curve shows the total average cost achieved by the learned policy with $\hat{\mathds{U}}=\{-1,0,1\}$ for different state-space quantization rates, while the lower curve shows the performance of the learned policies with $\hat{\mathds{U}}=\{-1,-0.5,0,0.5,1\}$ under varying state-space quantizations. For the state space, we use uniform quantization, i.e. if the size of the discrete state space is $M=3$, the quantization bins are $[0,\frac{1}{3}],(\frac{1}{3},\frac{2}{3}],(\frac{2}{3},1]$.
It is clear from the figure that as the quantization rate increases the regret decreases. Note that the asynchronous and the synchronous algorithms learn the same policy, therefore we do not plot them separately. In the same figure, we also plot the discretization error ($L_\mathds{X}$), and the accumulated average cost. As our results also suggest, the cost increases linearly with the increasing discretization error.
\begin{figure}[h!]
\begin{center}
\includegraphics[scale=0.5]{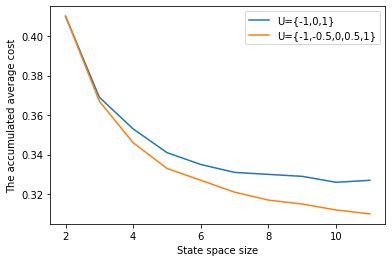}\includegraphics[scale=0.83]{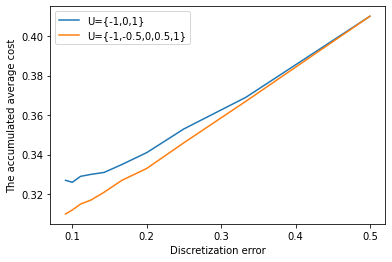}
\caption{Learned policy performance under different quantization rates}\label{fig2}
\end{center}
\end{figure}

\section{Concluding Remarks}

For infinite-horizon average-cost criterion problems, we presented approximation and reinforcement learning results for Markov Decision Processes with standard Borel spaces. We first provided a discretization based approximation method for fully observed Markov Decision Processes (MDPs) with continuous spaces under average cost criterion, and we derived error bounds for the approximations when the dynamics are only weakly continuous under certain ergodicity assumptions. In particular, we relaxed the total variation condition given in prior work to weak continuity as well as Wasserstein continuity conditions. We then presented synchronous and asynchronous Q-learning algorithms for continuous spaces via quantization, and established their convergence. We showed that the convergence is to the optimal Q values of the finite approximate models constructed via quantization. 

A direction for future research is the following. Since we are considering an average cost problem, one can obtain an online Q-learning algorithm where the past is used to optimally adapt the exploration policy, so that the optimal cost is obtained for each sample path under mild ergodicity conditions. This can be done, e.g., by increasing the exploration lengths and adapting the resulting policy using diminishing exploration.

\acks{
The authors are grateful to an anonymous referee and the Editor for their care, attention, and insightful and constructive comments.
}

\appendix

\section{Proof of Lemma \ref{holder_lemma}}\label{holder_lemma_proof}

\begin{proof}
We define iteratively 
\begin{align*}
h_{k+1}(x)&=\inf_{u\in\mathds{U}}\left\{c(x,u)+\int h_k(x_1)\mathcal{T}(dx_1|x,u)-\int h_k(x_1)\nu(dx_1)\right\}\\
&=\inf_{u\in\mathds{U}}\left\{c(x,u)+\kappa \int h_k(x_1){R}(dx_1|x,u)\right\}
\end{align*}
where  $R(\cdot|x,u):=\frac{\mathcal{T}(\cdot|x,u)-\nu(\cdot)}{\kappa}$, with $h_0\equiv 0$. Note that $h_k\to h$ uniformly.

We write for some $x,y\in\mathds{X}$:
\begin{align*}
\left|h_{k+1}(x)-h_{k+1}(y)\right|& \leq \sup_u\left\{ \left|c(x,u)-c(y,u)\right| + \kappa\left|\int h_k(x_1){R}(dx_1|x,u) - \int h_k(x_1){R}(dx_1|y,u)\right|\right\}\\
&\leq K_c d_{\mathds{X}}(x,y)^{1-\delta} d_{\mathds{X}}(x,y)^\delta + \kappa[h_k]_{H^\delta}\left[ \frac{K_\mathcal{T}}{\kappa} d_{\mathds{X}}(x,y)\right]^\delta\\
&\leq \left(K_c D^{1-\delta} + \kappa [h_k]_{H^\delta} \left(\frac{K_\mathcal{T}}{\kappa}\right)^\delta  \right)  d_{\mathds{X}}(x,y)^\delta.
\end{align*}
Noting $d_{\mathds{X}}(x,y) \leq D$, we obtain 
\begin{align*}
[h_{k+1}]_{H^\delta} \leq K_c D^{1-\delta} \sum_{t=0}^k \left( \kappa \left(\frac{K_\mathcal{T}}{\kappa}\right)^\delta  \right)^t \leq \frac{K_cD^{1-\delta}}{1-\kappa\left(\frac{K_\mathcal{T}}{\kappa} \right)^\delta}.
\end{align*}
Together, with the fact that $h_k\to h$ uniformly and that the H\"older constant is lower semicontinuous under uniform convergence, we conclude that
\begin{align*}
[h]_{H^\delta} \leq \frac{K_c D^{1-\delta}}{1-\kappa\left(\frac{K_\mathcal{T}}{\kappa}\right)^\delta}.
\end{align*}
\end{proof}

\sy{
\section{Proof of Theorem \ref{syn_th}}\label{proofSynQQL}
\begin{proof}
The main part of the proof is the convergence of the algorithm. If convergence is shown, then (i) follows from Proposition \ref{val_dif} and Lemma \ref{lip_lem}, (ii) follows from Proposition \ref{val_dif2} and Theorem \ref{pol_err} together with Lemma \ref{lip_lem}, and (iii) follows from Theorem \ref{weakContApprAsy}. Now, let 
\begin{align*}
F(Q)(y,u) := C(y,u) + \sum_{u,j}\min_{u_1} Q(y_1,u_1)P(y_1|y,u) .
\end{align*}
The process $Q_t$ satisfies the following form, with $S_t = Q_t - Q^*$:
\[S_{t+1}(y,u) = (1 - \alpha_t) S_t(y,u) + \alpha_t\bigg((F(Q_t)(y,u) - F(Q^{*})(y,u))  + \hat{\rho}+ w_t \bigg),\]
where $\alpha_t = \frac{1}{t}$. Furthermore $w_t: = \bigg(c(x,u)+  \min_{v} Q_t(Y_{1},v) - F(Q_t)(y,u) \bigg)$, where $x$ is generated from the bin $y$ belongs to according to the measure $\mu_y$. Note that $w_t$ is a zero-mean random variable. We will consider the following two parallel dynamics, as in \cite[Theorem 1]{jaakkola1994convergence}:
\begin{align}
S^a_{t+1}(y,u) &= (1 -\alpha_t(y,u)) S^a_t(y,u) + \alpha_t(y,u) w_t , \nonumber \\
S^b_{t+1}(y,u) &= (1 - \alpha_t(y,u)) S^b_t(y,u) + \alpha_t(y,u)\bigg(F(Q_t)(y,u) - F(Q_t^{*})(y,u)  +  \hat{\rho}\bigg) \label{denk2Bound},
\end{align}
We have $S_t(y,u) = S^a_{t+1}(y,u) + S^b_{t+1}(y,u)$. We will show further below that $S^a_{t+1}(y,u) \to 0$ almost surely. For now we assume this holds and focus on 
\begin{align*}
\|S^a_t + S^b_t\|_{sp}= \max_{y,u} (S^b_{t}(y,u) + S^a_{t}(y,u)) - \min_{y,u} (S^b_{t}(y,u) + S^a_{t}(y,u)).
\end{align*}
Under Assumption \ref{mix_kernel} $\|(F(Q_t)(\cdot,\cdot) - F(Q^{*})(\cdot,\cdot))\|_{sp} \leq \beta \|S_t\|_{sp} \leq \beta \|S^a_{t}\|_{sp} + \beta \|S^b_{t}\|_{sp}$, since $\|P(\cdot|y,u) - P(\cdot|y',u')\|_{TV} \leq \beta <1 $. With $\|S_t^a\|_\infty\to0$ almost surely, for every $\epsilon > 0$, there exists $N$ such that for $t \geq N$,  $\|S^a_{t+1}\|_{\infty} \leq \frac{\epsilon}{2}$ and so that $\|S^a_{t+1}\|_{sp} \leq \epsilon$ (where we suppress the sample path dependence). In the following, we assume that $t \geq N$. For some $M$ large enough, let $\hat{\beta} := \beta \frac{(M+1)}{M} < 1$. Furthermore, for $\|S^b_t\|_{sp} > M \epsilon$, we note that for any $(y',u') \in \mathds{Y} \times \mathbb{U}$
\[\beta \|S^b_t(y,u) - S^b_t(y',u') + \epsilon\| \leq \hat{\beta} \|S^b_t\|_{sp}.\]
For $(y',u') \in \mathds{Y} \times \mathds{U}$, $S^b_{t+1}(y',u') = (1 - \alpha_t) S^b_t(y',u') + \alpha_t\bigg(F(Q_t)(y',u') - F(Q^{*})(y',u') + \hat{\rho}\bigg)$: 
\begin{align}
& |S^b_{t+1}(y,u) - S^b_{t+1}(y',u')| \leq  (1 - \alpha_t) |S^b_t(y,u) - S^b_{t}(y',u')| \nonumber \\
&\quad + \bigg|\alpha_t\bigg(F(Q_t)(y,u) - F(Q^{*})(y,u)  -  \bigg(F(Q_t)(y',u') - F(Q^{*})(y',u') \bigg) \bigg|  \nonumber \\ 
& \leq  (1 - \alpha_t) \|S^b_t\|_{sp} + \alpha_t (\beta \|S^a_{t}\|_{sp} + \beta \|S^b_{t}\|_{sp}) \label{boundQ1}\\
& \leq  (1 - \alpha_t)  \|S^b_t\|_{sp} + \alpha_t \hat{\beta} \|S^b_t\|_{sp} = \left[1-\alpha_t(1-\hat{\beta})\right]\|S^b_t\|_{sp}\label{boundQ2}
\end{align}
In particular, opening the last bound iteratively, and noting that $\alpha_t=\frac{1}{t+1}$, we can write that for some $l\geq1$ we can write
$\|S_{t+l}^b\|_{sp}\leq  \|S_t^b\|_{sp} \prod_{k=t}^{t+l} (1-\frac{1-\hat{\beta}}{k})$. This product can be made arbitrarily small for any $t$ by choosing $l$ large enough. Therefore, there exists some $l<\infty$, such that $\|S^b_{t+l}\|_{sp}\leq M\epsilon$.  Furthermore, once $\|S^b_t\|_{sp} \leq M \epsilon$, we can show via (\ref{boundQ1}) and $\beta (M+1)/M < 1$ that it will remain there thereafter. Thus, for any $\epsilon > 0$, for large enough $t$, we have that $\|S^b_t\|_{sp} \leq M \epsilon$. Since $\epsilon > 0$ is arbitrary, the convergence result follows.  

We now discuss $S^a_t$. Taking the square of $S^a_t$, we obtain:
\begin{align}\label{equationBF}
E[(S^a_{t+1}(y,u))^2| {\cal F}_t] \leq (S^a_t(y,u))^2 - 2 \alpha_t (S^a_t(y,u))^2+ \alpha_t^2 (S^a_t(y,u))^2 +  \alpha_t^2 w_t^2 
\end{align}
First, we have that for any $T > 0$,
\begin{eqnarray}\label{comparisonQ}
E[\sum_{t=0}^{T-1} (2 \alpha_t - \alpha_t^2) (S^a_t(y,u))^2 &\leq& (S^a_0(y,u))^2 + E[\sum_{t=0}^{T-1} \alpha_t^2 w_t^2] 
\end{eqnarray}
We now show that the right hand term is bounded. Consider:
\[ Q_{t+1}(y,u) = (1 - \alpha_t) Q_{t}(y,u) +   \alpha_t (c(x,u)+ \min_{v} Q_t(q(x_{1}),v))\]
which implies that
\begin{align} |Q_{t+1}(y,u)| &\leq (1 - \alpha_t) \|Q_{t}\|_{\infty} +   \alpha_t (c(x,u)+  \|Q_{t}\|_{\infty} )  \leq \|Q_{t}\|_{\infty} + \alpha_t \|c\|_{\infty}, \nonumber 
\end{align}
And thus, since this holds for all $(y,u)$ pairs, $\|Q_{t+1}\|_{\infty} \leq \|Q_{t}\|_{\infty} + \alpha_t \|c\|_{\infty}$. By bounding the partials sums of harmonic series $\sum_{k=1}^t \frac{1}{k}$, the above implies that $\|Q_{t+1}\|_{\infty} \leq \log(t) \|c\|_{\infty} + M_1,$ for some finite $M_1$. However, $\alpha_t^2 w_t^2 \leq (1/t)^2 \bigg(2 \|c\|_{\infty}^2 + 4(\log(t) \|c\|_{\infty} + M_1)^2\bigg)$ is a summable sequence, and the right hand side of (\ref{comparisonQ}) is bounded. Furthermore, re-writing (\ref{equationBF}), in the expression
\[E[(S^a_{t+1}(y,u))^2| {\cal F}_t] \leq (S^a_t(y,u))^2 - (2 \alpha_t - \alpha_t^2) (S^a_t(y,u))^2 +  \alpha_t^2 w_t^2,\]
the term $\alpha_t^2 w_t^2$ is finite almost surely. This implies, by \cite[p. 33, Exercise II-4]{neveu1975discrete} (see also \cite[Exercises 4.4.13 and  4.4.14]{yuksel2020control} for a more explicit discussion), that $S^a_t$ converges to some random variable almost surely. The finiteness of the right hand term in (\ref{comparisonQ}) then implies that the limit must be zero: Suppose not; as $\alpha_t$ is not summable, there exists an infinite sequence of times so that each summation of  $\alpha_t$ between the times is bounded from below by a positive constant. Through this, if $(S^a_t)^2$ were not to converge to zero (given that it does converge to something), it would remain above a positive constant after a sufficiently large time, and then it would follow that $\sum_t (2 \alpha_t - \alpha_t^2 ){S^a_t}^2$ would not remain bounded. Therefore, if this were to happen with non-zero measure, the expectation would be unbounded, which in turn would, as $T \to \infty$, violate (\ref{comparisonQ}).

Finally, if the sequence converges under the span semi-norm, for some constant $\hat{\rho}$.
\[\hat{\rho} + {Q}^*(y,u)=F({Q}^*)(y,u) = C(y,u) + \sum_{y^{\prime}}P(y^{\prime} | y, u)\min_{v} {Q}^*(t^{\prime},v)\]
Note that the minimum of $u$, for each $y$, is the solution to the Average Cost Optimality Equation. Hence, the stationary policy $\{f^*\}$ is optimal. Furthermore, $\hat{Q}^*$ is only a constant shifted version of $Q^*$, $\hat{Q}^*$ also satisfies the ACOE. By definition, we have $\hat{Q}^*(y_0,u_0)=0$. Thus, we have that 
\begin{align*}
\hat{\rho}+\hat{V}^*(y_0)=C(y_0,u_0)+\sum_{y_1\in\mathds{Y}} \hat{V}^*(y_1) P(y_1|y_0,u_0).
\end{align*}
\end{proof}
}

\section{On Lemma \ref{czaba_lemma}}\label{discussionLemmaCsaba}

We provide a short proof for Lemma \ref{czaba_lemma} as applied to our theorem, essentially building on \cite{BertsekasTsitsiklis}. Write
\begin{align*}
\Delta_{t+1}(x)=(1-\alpha_t(x))\Delta_t(x)+\alpha_t(x)F_t(x)
\end{align*}
as
\begin{align*}
\Delta_{t+1}(x)=(1-\alpha_t(x))\Delta_t(x)+\alpha_t(x) \bigg(E[F_t(x)|P_t] +  F_t(x) - E[F_t(x)|P_t] \bigg)
\end{align*}
We will take $B_t$ so that $\|\Delta_t\|_{\infty} \leq B_t$ with $\hat{w}_t := \frac{F_t(x) - E[F_t(x)|P_t]}{B_t}$ so that $E[\hat{w}^2_t] \leq K$ by assumption.
To this end, we take $B_0 = 1+\| \Delta_{0}\|_{\infty}$ and for $t \in \mathbb{Z}_+$,
\[B_{t+1} = \max(B_t, 1+\| \Delta_{t+1}\|_{\infty})\]
Let
\[R_t(x) = \frac{\Delta_t(x)}{B_t}.\]
Write
\[\Delta_{t+1}(x) \leq (1-\alpha_t(x))B_t R_t(x)+\alpha_t(x) \bigg(E[F_t(x)|P_t] +  B_t \hat{w}_t \bigg)\]
and
\begin{align}\label{denk11}
\Delta_{t+1}(x) \leq B_t \bigg( (1-\alpha_t(x)) R_t(x)+\alpha_t(x) \bigg(\beta R_t(x) +  \hat{w}_t \bigg) \bigg)
\end{align}
Therefore,
\[R_{t+1}(x)B_{t+1} \leq B_t \bigg( (1-\alpha_t(x)) R_t(x)+\alpha_t(x) \bigg(\beta R_t(x) +  \hat{w}_t \bigg) \bigg)\]
and since $\frac{B_t}{B_{t+1}}\leq 1$,
\[R_{t+1}(x) \leq \bigg( (1-\alpha_t(x)) R_t(x)+\alpha_t(x) \bigg(\beta R_t(x) +  \hat{w}_t \bigg) \bigg)\]
We know that $(1-\alpha_t(x)) R_t(x)+\alpha_t(x) \bigg(\beta R_t(x) +  \hat{w}_t \bigg) \to 0$ almost surely (\cite{TsitsiklisQLearning} or \cite[Theorem 9.1.1]{yuksel2020control}) as $\hat{w}_t$ has a uniformly bounded variance. Accordingly, for large enough $t$, 
\[R_{t+1}(x) \leq \epsilon,\]
and by (\ref{denk11})
\[B_{t+1} = \max(B_t, 1+ B_t \epsilon_t)\]
This implies then that $B_t$ is bounded (almost surely; though not necessarily by a constant over all realizations), and so is $\Delta_t$. Once we have that $\Delta_t$ is bounded, we can write
\begin{align*}
S^b_{t+1}(x)&=(1-\alpha_t(x))S^b_t(x)+\alpha_t(x) \bigg(E[F_t(x)|P_t] \bigg) \nonumber \\
S^a_{t+1}(x)&=(1-\alpha_t(x))S^a_t(x)+\alpha_t(x) \bigg(F_t(x) - E[F_t(x)|P_t]\bigg)
\end{align*}
Now, we have that in $S^a_t$ above, the term $\bigg(F_t(x) - E[F_t(x)|P_t]\bigg)$ has a conditionally bounded covariance, even though there is no uniform bound. Nonetheless, \cite[Corollary 4.1]{BertsekasTsitsiklis} implies that $S^a_t \to 0$ in this case as well. With $S^a_t \to 0$, $S_b(t) \to 0$ also via \cite{TsitsiklisQLearning} or \cite[Theorem 9.1.1]{yuksel2020control}.

\section{Proof of Theorem \ref{quant_q}}\label{quant_q_proof}

\begin{proof}
For small enough $\delta$, we define a positive measure $\nu'$ such that
\begin{align*}
\nu'(y_j)=\delta, 
\end{align*}
and $\nu'(y)=0$ otherwise where $y_j$ is as in Assumption \ref{mix_kernel2}.
$\nu'$ then satisfies 
\begin{align*}
P(\cdot|y,u)\geq\nu'(\cdot)
\end{align*}
since we selected $\delta<\min_{y,u} P(y_j|y,u)$.
We now define the following new transition measure (though, not a conditional probability measure)
\begin{align*}
\hat{P}(\cdot|y,u)=P(\cdot|y,u)-\nu'(\cdot).
\end{align*}
Then, as noted earlier, the following operator is a contraction with $1-\delta$ on bounded measurable functions $\B(\mathds{Y}\times{\mathds{U}})$ (see (\ref{contractAverageCostArgum}) and \cite[p.61]{hernandez2012adaptive}) 
\begin{align*}
(\hat{\mathbb{T}}f)(y,u)=C(y,u) + \sum_{y'}\min_uf(y',u)\hat{P}(y'|y,u).
\end{align*}

We define the fixed point, say $Q^*(y,u),$ of the mapping $\hat{\mathbb{T}}$, such that 
\begin{align*}
Q^*(y,u)&=(\hat{\mathbb{T}}Q^*)(y,u)=C(y,u)+\sum_{y'}\min_uQ^*(y',u)\hat{P}(y'|y,u)\\
&=C(y,u)+\sum_{y'}\min_uQ^*(y',u)P(y'|y,u)-\sum_{y'}\min_uQ^*(y',u)\nu'(y')\\
&=C(y,u)+\sum_{y'}\min_uQ^*(y',u)P(y'|y,u)-\delta\min_uQ^*(y_j,u)
\end{align*}
which satisfies an ACOE with $\delta\min_uQ^*(y_j,u)=\hat{\rho}$. 

We claim that the iterations converge to $Q^*$. Now, let $\Delta(y,u):=Q_t(y,u)-Q^*(y,u)$, we then write
\begin{align*}
\Delta_{t+1}(y,u)=(1-\alpha_t(y,u))\Delta_t(y,u)+\alpha_t(y,u)F_t(y,u)
\end{align*}
where
\begin{align*}
F_t(y,u)=&\left(c(X_t,U_t)+  V_t(Y_{t+1}) - \delta V_t(y_j)\right) \\
&-\left(C(y,u)+\sum_{y'}V^*(y')P(y'|y,u)-\delta V^*(y_j)\right)
\end{align*}
where $Y_{t+1}=q(X_{t+1})$. We now write $F_t=\hat{F}_t+r_t$ by adding and subtracting $V^*(Y_{t+1})$ with
\begin{align}
&\hat{F}_t(y,u)=V_t(Y_{t+1})- \delta V_t(y_j) - \left(V^*(Y_{t+1})- \delta V^*(y_j)\right) \nonumber \\
&r_t(y,u)=c(X_t,U_t)-C(y,u)+V^*(Y_{t+1})- \sum_{y'}V^*(y')P(y'|y,u). \label{rhatF}
\end{align}
We define two processes $\delta_t(y,u)$ and $v_t(y,u)$:
\begin{align*}
&\delta_{t+1}(y,u)=(1-\alpha_t(y,u))\delta_t(y,u)+\alpha_t(y,u)\hat{F}_t(y,u) \\
&v_{t+1}(y,u)=(1-\alpha_t(y,u))v_t(y,u)+\alpha_t(y,u)r_t(y,u).
\end{align*}

We first show that $v_t(y,u)\to 0$ for all $(y,u)$. Due to how we have chosen the learning rates, one can show that
\begin{align*}
v_{t+1}(y,u)=\frac{\sum_{k=0}^{t-1} r_{k}(y,u) \mathds{1}_{\{(Y_k,U_k)=(y,u)\}}}{\sum_{k=0}^{t-1} \mathds{1}_{\{(Y_k,U_k)=(y,u)\}}}
\end{align*}
where $Y_k=q(X_k)$. 
By Items~(2) and~(3) of Assumption~\ref{asyn_q}, the stationary probability of the event $\{(Y_t,U_t)=(y,u)\}$ equals $\mu_\pi(B_y)\cdot\pi(u|y)>0$ for every $(y,u)\in\mathds{Y}\times\mathds{U}$. Hence every pair $(y,u)$ is visited infinitely often almost surely, so the denominator above diverges to infinity as $t\to\infty$ and the ratio is well-defined for all sufficiently large $t$. This property also ensures that the step size conditions $\sum_t \alpha_t(y,u)=\infty$ and $\sum_t \alpha_t^2(y,u)<\infty$ of Lemma~\ref{czaba_lemma} hold almost surely for every $(y,u)$. Note that the joint process $\{X_t,Y_t,U_t\}$ is Markovian and has a unique stationary measure under Assumption \ref{asyn_q}  given by
\begin{align*}
Pr(X_t\in A,Y_t=y,U_t=u) =\int_A  \pi(u|y)\mathds{1}_{\{q(X_t)=y\}}\mu_\pi(dx)
\end{align*}
for any $A\in\B(\mathds{X})$, and for any $y,u\in\mathds{Y}\times\mathds{U}$ where $\mu_\pi$ is the stationary distribution of the hidden state process under the exploration policy $\pi$.

 Thus, we first write assuming $y\in B_i$
\begin{align*}
&\frac{\sum_{k=0}^{t-1} c(X_k,U_k) \mathds{1}_{\{(Y_k,U_k)=(y,u)\}}}{\sum_{k=0}^{t-1} \mathds{1}_{\{(Y_k,U_k)=(y,u)\}}} = \frac{\frac{1}{t}\sum_{k=0}^{t-1} c(X_k,U_k) \mathds{1}_{\{(Y_k,U_k)=(y,u)\}}}{\frac{1}{t}\sum_{k=0}^{t-1} \mathds{1}_{\{(Y_k,U_k)=(y,u)\}}}\\
 &\to \frac{\int_{B_i} c(x,u)\pi(u|y_i)\mu_\pi(dx)   }{ \pi(u|y_i) \mu_\pi(B_i)} = \frac{\int_{B_i}c(x,u)\mu_\pi(dx)}{\mu_\pi(B_i)}= C(y,u).
\end{align*}
Using identical arguments, we can also show that
\begin{align*}
&\frac{\sum_{k=0}^{t-1} V^*(Y_{t+1}) \mathds{1}_{\{(Y_k,U_k)=(y,u)\}}}{\sum_{k=0}^{t-1} \mathds{1}_{\{(Y_k,U_k)=(y,u)\}}} \to \sum_{y'}V^*(y')P(y'|y,u).
\end{align*}
Hence, we have that $v_t(y,u)\to 0$ almost surely for all $(y,u)\in\mathds{Y}\times\mathds{U}$.
To show that $\delta_t$ converges to zero, we will use Lemma \ref{czaba_lemma}.
Under Assumption \ref{mix_kernel2}
\begin{align*}
E[\hat{F}_t(y_i,u_i)|h_t]\leq\beta \|V_t-V^*\|_\infty\leq \beta\|\Delta_t\|_\infty \leq  \beta\|\delta_t\| + \beta\|v_t\|_\infty
\end{align*}
where $\beta:=(1-\delta)<1$, and $\|v_t\|_\infty$ converges to zero almost surely. We finally need to verify that 
\begin{align*}
Var(\hat{F}_t|h_t)\leq K(1+\|\delta_t\|_\infty)^2
\end{align*}
where $h_t=\{y_t,u_t,\dots,y_0,u_0\}$. We write
\begin{align*}
Var(\hat{F}_t|h_t)=&E\bigg[ \bigg(  V_t(Y_{t+1}) -\delta V_t(y_j) - V^*(Y_{t+1}) + \delta V^*(y_j)   \\
& - \int V_t(y')P(y'|h_t) + \delta V_t(y_j) + \int V^*(y')P(y'|h_t) -  \delta V^*(y_j)  \bigg)^2\bigg]\\
&= E\bigg[ \bigg(   V_t(Y_{t+1}) -  V^*(Y_{t+1}) - \int V_t(y')P(y'|h_t) + \int V^*(y')P(y'|h_t)\bigg)^2\bigg]\\
&\leq  \|V_t - V^*\|^2_\infty \leq \|\delta_t + v_t\|^2_\infty.
\end{align*}
Note that $\|v_t\|_\infty$ remains bounded uniformly over $t$ (and over all sample paths as $V^*, c$ are bounded in (\ref{rhatF})), since the cost function $c(x,u)$ and $V^*(y)$ are uniformly bounded and thus $r_t(y,u)$ is bounded.  Hence, we can find some $\hat{K}$ such that 
\begin{align*}
Var(\hat{F}_t|h_t) \leq \hat{K}(1+\|\delta_t\|_\infty)^2.
\end{align*}
Thus, we can conclude that $\|\delta_t\|_\infty\to 0$ almost surely.

\end{proof}

\section{Proof of Theorem \ref{young_conv}}\label{young_conv_proof}

\begin{proof} To make the analysis explicit let ${\cal T}_n$ be the discretized finite state/action model given by (\ref{finite_cost}). Notably, let $\pi_n$ be optimal for ${\cal T}_n$. We then seek to bound:

\[J({\cal T}, \pi_n) - J({\cal T}_n, \pi_n)  + J({\cal T}_n, \pi_n) - J({\cal T}, \pi) \]

The second term above goes to zero by Proposition \ref{val_dif}. We consider then the first term
	\[J({\cal T}, \pi_n) - J({\cal T}_n, \pi_n)\]
Instead of working with the ACOE, we will find it convenient to work with the induced invariant probability measures. Let the invariant measures be: for ${\cal T}$ under $\pi_n$ be $\nu^n$ (with $\nu^n \ll \psi$); and for ${\cal T}_n$ under $\pi_n$ be $\hat{\nu}^n$, so that

\begin{align}\label{nunT}
	\nu^n(dx') = \int \nu^n(dx) \int {\cal T}(dx’|x , u) \pi_n(du|x)
	\end{align}

\begin{align}\label{nunT^nn}
\hat{\nu}^n(dx') = \int \hat{\nu}^n(dx) \int {\cal T}_n(dx’|x , u) \pi_n(du|x)
	\end{align}
where (\ref{nunT^nn}) is constructed via the invariance equation
\begin{align}\label{nunT^n}
\hat{\nu}^n(B^n_{y'}) = \sum_{y} \hat{\nu}^n(B^n_y) \int \hat{\nu}^n(dx | y) \int  ({\cal T}(B^n_{y'}|s , u)  \mu^n_{y}(ds)) \pi_n(du|x).
	\end{align}
Here, $\mu^n_y$ is a pre-defined weighting measure and $\hat{\nu}^n(dx | y)$ is an artificial normalization as it does not impact the integration on the right hand side noting that $\pi_n(du|x)$ is a constant policy for all $x \in B^n_y$. Note that by the construction above as ${\cal T}(\cdot|s , u) \ll \psi(\cdot)$, we have that we can extend $\hat{\nu}^n(B^n_{y'})$ to the entire $\mathbb{X}$ so that $\hat{\nu}^n(dx) := \sum_{y} \hat{\nu}^n(B^n_y) \int \hat{\nu}^n(dx | y) \ll \psi(dx)$.

The question is, does 
\[	\int \nu^n(dx) c(x,\pi_n(x))  - \int \hat{\nu}^n(dx) c(x,\pi_n(x)) \to 0\]
as $n \to \infty$? Consider (\ref{nunT}) and take $n \to \infty$. We have that $\nu^n$ to some $\bar{\nu}$ along a subsequence weakly by compactness, where this convergence is in total variation by equi-continuity of densities with respect to $\psi$. Furthermore, along a further subsequence, $\pi_n \to \pi^*$ for some $\pi^*$ in Young topology at input $\psi$ which then implies that this is also true at input $\bar{\nu}$ by \cite[Lemma 3.6]{yuksel2023borkar}. Since convergence of $\nu^n(dx) \to \bar{\nu}(dx)$ is also in total variation \cite[Theorem 4.1]{yuksel2023borkar}, this then implies that $\nu^n(dx)\pi_n(du|x)$ converges weakly to some $\bar{\nu}(dx) \pi^*(du|x)$.

For the second term (\ref{nunT^n}), we note that $\hat{\nu}^n(dx)$ is defined first by its discrete support on the bins $B^n_y$ via the weighting measures $\mu_y$, 
but this can be extended to the entire $\mathbb{X}$ so that $\hat{\mu}^n \ll \psi$ as noted above. Therefore, by the equi-continuity condition, this sequence of measures will have a converging subsequence which does so in total variation \cite[Theorem 4.1]{yuksel2023borkar} to some measure $\tilde{\nu} \ll \psi$ in total variation. This then implies that the measure $\hat{\nu}^n(dx)\pi_n(du|x)$ converges weakly, along a subsequence, to some $\tilde{\nu}(dx) \pi^*(du|x)$, 

Since ${\cal T}$ is weakly continuous, it follows that for (\ref{nunT^n}), by \cite[Theorem 3.5]{serfozo1982convergence} or \cite[Theorem 3.5]{Lan81}, the limit leads to invariance and therefore the limits $\bar{\nu}$ and $\tilde{\nu}$ have to be  invariant under the same policy $\pi^*(du|x)$. However, since the invariant measure is to be unique given the policy $\pi^*$ by hypothesis, it must be that these limit measures are identical. Therefore, the induced costs \[\int  \nu^n(dx) \pi_n(du|x) c(x,u) - \int \hat{\nu}^n(dx) \pi_n(du|x) c(x,u) \to 0.\]

It should be noted that $\pi_n$ converges to a limit $\pi^*$ under the Young topology (and not pointwise in $x$); this is why the analysis above is needed.  One could also note that the analysis above can be generalized to the case where $\mathbb{X} = \mathbb{R}^d$, however with the required tightness conditions on the set of invariant measures as in \cite[Theorem 4.1]{yuksel2023borkar} and the associated ACOE conditions. 

\end{proof}

\section{Alternative Ergodicity Conditions}\label{erg_conds_sec}
\begin{proposition}\label{erg_conds}\cite[Theorem 3.2]{hernandez1991recurrence}
Consider the following.
\begin{itemize}
\item [a.] There exists a state $x^*\in\mathbb{X}$ and a number $\beta>0$ such that $\mathcal{T}(\{x^*\}|x,\pi)\geq\beta$, for all  $x\in\mathbb{X},\pi\in\Pi_s$.
\item[b.] There exists a positive integer $t$ and a non-trivial measure $\nu$ on $\mathbb{X}$ such that $\mathcal{T}^t(\cdot|x,\pi)\geq\nu(\cdot)$ for all $x\in\mathbb{X},\pi\in\Pi_s$.
\item [c.] For each $\pi\in\Pi_s$, the transition kernel $\mathcal{T}(dy|x,\pi)$ has a density $p(y|x,\pi)$ with respect to a sigma-finite measure $m$ on $\mathbb{X}$, and there exist $\epsilon>0$ and $C\in\mathcal{B}(\mathbb{X})$ such that $m(C)>0$ and $p(y|x,\pi)\geq\epsilon$ for all $y\in C, x\in \mathbb{X}, \pi\in\Pi_s$.
\item[d.] For each $\pi\in\Pi_s$, $\mathcal{T}(dy|x,\pi)$ has a density $p(y|x,\pi)$ with respect to a sigma-finite measure $m$ on $\mathbb{X}$, and $p(y|x,\pi)\geq p_0(y)$ for all $x,y\in\mathbb{X}$, $\pi\in\Pi_s$, where $p_0$ is a non-negative measurable function with $\int p_0(y)m(dy)>0$.
\item[e.] There exists a positive integer $t$ and a measure $\nu$ on $\mathbb{X}$ such that $\nu(\mathbb{X})<2$ and $\mathcal{T}^t(\cdot|x,\pi)\leq\nu(\cdot)$ for all $x\in\mathbb{X}, \pi\in\Pi_s$.
\item[f.] There exists a positive integer $t$ and a positive number $\beta<1$ such that $\|\mathcal{T}^t(\cdot|x,\pi)-\mathcal{T}^t(\cdot|x',\pi)\|_{TV}\leq 2\beta$ for all $x,x'\in\mathbb{X}, \pi\in\Pi_s$.
\item[g. ] There exists a positive integer $t$ and a positive number $\beta$ for which the following holds: For each $\pi\in\Pi_s$, there is a probability measure $\nu_\pi$ on $\mathbb{X}$ such that $\mathcal{T}^t(\cdot|x,\pi)\geq \beta\nu_\pi(\cdot)$ for all $x\in \mathbb{X}$.
\item[h. ] There exist positive numbers $c$ and $\beta$, with $\beta<1$, for which the following holds: For each $\pi\in\Pi_s$, there is a probability measure $p_\pi$ on $\mathbb{X}$ such that $\|\mathcal{T}^t(\cdot|x,\pi)-p_\pi(\cdot)\|_{TV}\leq c\beta^t$ for all $x\in\mathbb{X}, t\in\N$.
\item[i. ] The state process is uniformly ergodic such that 
$\lim_{t\to\infty}\|\mathcal{T}^t(\cdot|x,\pi)-p_\pi(\cdot)\|_{TV}=0$
uniformly in  $x\in\mathbb{X}$  and $\pi\in\Pi_s$.
\end{itemize}
The conditions above are related as follows:
\begin{align*}
&a \to b\\
&e\to f\\
&c\to d \to b  \to f \leftrightarrow g \leftrightarrow h \leftrightarrow i.
\end{align*}
\end{proposition}

\bibliography{SerdarBibliography,AliBibliography}

\end{document}